\documentclass[10pt,psamsfonts]{article}

\usepackage{amsmath, amsfonts, amsthm, amssymb, enumerate, url, float, slashed, MnSymbol}
\usepackage[all,arc,color]{xy}
\usepackage{enumerate}
\usepackage{mathrsfs}
\usepackage{mathtools}
\usepackage[utf8]{inputenc}
\usepackage{tikz-cd}

\begingroup
    \makeatletter
    \@for\theoremstyle:=definition,remark,plain\do{%
        \expandafter\g@addto@macro\csname th@\theoremstyle\endcsname{%
            \addtolength\thm@preskip\parskip
            }%
        }
 \endgroup
 
   \usepackage[titletoc,toc,title]{appendix}
 
   \usepackage[colorlinks = true,
            linkcolor = blue,
            urlcolor  = blue,
            citecolor = blue,
            anchorcolor = blue]{hyperref}
            
   \usepackage[parfill]{parskip}
   \usepackage{anysize}
   \marginsize{1in}{1in}{1in}{1in}
   \usepackage{thmtools}

\declaretheorem[name=Theorem,numberwithin=section]{thm}
\declaretheorem[name=Proposition,numberlike=thm]{prop}
\declaretheorem[name=Lemma,numberlike=thm]{lemma}
\declaretheorem[name=Corollary,numberlike=thm]{cor}

\declaretheorem[name=Definition,style=definition,qed=$\blacktriangle$,numberlike=thm]{defn}
\declaretheorem[name=Example,style=definition,qed=$\blacktriangle$,numberlike=thm]{ex}
\declaretheorem[name=Remark,style=definition,qed=$\blacktriangle$,numberlike=thm]{rmk}

\newcounter{commentCounter}
\newcommand{\bu}{\bullet}
\newcommand{\dd}{\mathrm{d}}
\newcommand{\sta}{\star}
\newcommand{\dop}[1]{{#1}{}^{\sta}}
\newcommand{\ds}{\dop{\dd}}

\newcommand{\vol}{\mathsf{vol}}
\newcommand{\real}{\operatorname{Re}}
\newcommand{\imag}{\operatorname{Im}}
\newcommand{\G}{\mathrm{G}_2}
\newcommand{\Spin}[1]{\mathrm{Spin}(#1)}
\newcommand{\GL}[1]{\mathrm{GL}(#1)}

\newcommand{\U}[1]{\mathrm{U}(#1)}
\newcommand{\SU}[1]{\mathrm{SU}(#1)}
\newcommand{\SO}[1]{\mathrm{SO}(#1)}
\newcommand{\Or}[1]{\mathrm{O}(#1)}
\newcommand{\Sp}[1]{\mathrm{Sp}(#1)}

\newcommand{\Hol}{\mathsf{Hol}}

\newcommand{\A}{\mathbb A}
\newcommand{\R}{\mathbb R}
\newcommand{\C}{\mathbb C}

\newcommand{\Qu}{\mathbb H}
\newcommand{\Oc}{\mathbb O}
\newcommand{\Z}{\mathbb Z}
\newcommand{\PR}{\mathbb P}

\newcommand{\ph}{\varphi}
\newcommand{\ps}{\psi}
\newcommand{\st}{\ast}
\newcommand{\hk}{\mathbin{\! \hbox{\vrule height0.3pt width5pt depth 0.2pt \vrule height5pt width0.4pt depth 0.2pt}}}

\newcommand{\tr}{\operatorname{Tr}}

\newcommand{\ddx}[1]{\frac{\del}{\del x^{#1}}}

\newcommand{\dx}[1]{d x^{#1}}
\newcommand{\nab}[1]{\nabla_{#1}}
\newcommand{\del}{\partial}

\newcommand{\ddt}{\frac{d}{d t}}

\newcommand{\ol}{\overline}

\newcommand{\spa}{\operatorname{span}}

\newcommand{\End}{\operatorname{End}}
\newcommand{\spi}{\slashed{\mathcal{S}}}

\newcommand{\sym}{\mathrm{S}^2 (T^* M)}

\newcommand{\Sym}{\mathcal{S}}
\newcommand{\Symo}{\mathcal{S}_0}
\newcommand{\oo}{\mathrm{o}}

\makeatletter
\let\c@equation\c@thm
\makeatother
\numberwithin{equation}{section}

\begin{document}

\title{Introduction to $\G$ geometry}

\author{Spiro Karigiannis \\ {\it Department of Pure Mathematics, University of Waterloo} \\ \tt{karigiannis@uwaterloo.ca}}

\date{September 20, 2019}

\maketitle

\begin{abstract}
These notes give an informal and leisurely introduction to $\G$ geometry for beginners. A special emphasis is placed on understanding the special linear algebraic structure in $7$ dimensions that is the pointwise model for $\G$ geometry, using the octonions. The basics of $\G$-structures are introduced, from a Riemannian geometric point of view, including a discussion of the torsion and its relation to curvature for a general $\G$-structure, as well as the connection to Riemannian holonomy. The history and properties of torsion-free $\G$ manifolds are considered, and we stress the similarities and differences with K\"ahler and Calabi-Yau manifolds. The notes end with a brief survey of three important theorems about compact torsion-free $\G$ manifolds.
\end{abstract}

\tableofcontents

\section{Aim and scope} \label{sec:aim-scope}

The purpose of these lecture notes is to give the reader a gentle introduction to the basic concepts of $\G$ geometry, including a brief history of the important early developments of the subject.

At present, there is no ``textbook'' on $\G$ geometry. (This is on the author's to-do list for the future.) The only references are the classic monograph by Salamon~\cite{Salamon} which emphasizes the representation theoretic aspects of Riemannian holonomy, and the book by Joyce~\cite{Joyce} which serves as both a text on K\"ahler and Calabi-Yau geometry as well as a monograph detailing Joyce's original construction~\cite{J12} of compact manifolds with $\G$ Riemannian holonomy. Both books are excellent resources, but are not easily accessible to beginners. The book by Harvey~\cite{Harvey} is at a more appropriate level for new initiates, but is much broader in scope, so it is less focused on $\G$ geometry. Moreover, both~\cite{Salamon} and~\cite{Harvey} predate the important analytic developments that started with Joyce's seminal contributions.

The aim of the present notes is to attempt to at least psychologically prepare the reader to access the recent literature in the field, which has undergone a veritable explosion in the past few years. The proofs of most of the deeper results are only sketched, with references given to where the reader can find the details, whereas most of the simple algebraic results are proved in detail. Some important aspects of $\G$ geometry are unfortunately only briefly mentioned in passing, including the relations to $\Spin{7}$-structures and the intimate connection with spinors and Clifford algebras. Good references for the connection with spinors are Harvey~\cite{Harvey}, Lawson--Michelsohn~\cite[Chapter IV. 10]{LM}, and the more recent paper by Agricola--Chiossi--Friedrich--H\"oll~\cite{ACFH}.

These notes are written in a somewhat informal style. In particular, they are meant to be read leisurely. The punchline is sometime spoiled for the benefit of motivation. In addition, results are sometimes explained in more than one way for clarity, and results are not always stated in the correct logical order but rather in an order that (in the humble opinion of the author) is more effective pedagogically. Finally, there is certainly a definite bias towards the personal viewpoint of the author on the subject. In fact, a distinct emphasis is placed on the explicit details of the linear algebraic aspects of $\G$ geometry that are consequences of the \emph{nonassociativity of the octonions $\Oc$}, as the author believes that this gives good intuition for the striking differences between $\G$-structures and $\U{m}$-structures in general and $\SU{m}$-structures in particular.

The reader is expected to be familiar with graduate level smooth manifold theory and basic Riemannian geometry. Some background in complex and K\"ahler geometry is helpful, especially to fully appreciate the distinction with $\G$ geometry, but is not absolutely essential.

\subsection{History of these notes} \label{sec:notes}

These lecture notes have been gestating for many years. In their current form, they are a synthesis of lecture notes for several different introductions to $\G$ and $\Spin{7}$ geometry that have been given by the author at various institutions or workshops over the past decade. Specifically, these are the following, in chronological order:
\begin{itemize} \setlength\itemsep{-1mm}
\item October 2006: Seminar; Mathematical Sciences Research Institute; Berkeley.
\item November 2008: Seminar series; University of Oxford.
\item January/February 2009: Seminar series; University of Waterloo.
\item August 2014: `\emph{$\G$-manifolds}'; Simons Center for Geometry and Physics; Stony Brook.
\item September 2014: `\emph{Special Geometric Structures in Mathematics and Physics}'; Universit\"at Hamburg.
\item August 2017: Minischool on `\emph{$\G$-manifolds and related topics}'; Fields Institute; Toronto.
\end{itemize}
The current version of these notes is the first part of the ``minischool lectures'' on $\G$-geometry collected in the book \emph{Lectures and Surveys on $\G$-geometry and related topics}, published in the \emph{Fields Institute Communications} series by Springer. The other parts of the minischool lectures are
\begin{itemize} \setlength\itemsep{-1mm}
\item ``Constructions of compact $\G$-holonomy manifolds'' by Alexei Kovalev~\cite{Minischool-Kovalev}
\item ``Geometric flows of $\G$ structures'' by Jason Lotay~\cite{Minischool-Lotay-flows}
\item ``Calibrated Submanifolds in $\G$ geometry'' by Ki Fung Chan and Naichung Conan Leung~\cite{Minischool-Leung}
\item ``Calibrated Submanifolds'' by Jason Lotay~\cite{Minischool-Lotay-calib}
\end{itemize}

\subsection{Notation} \label{sec:notation}

Let $(M, g)$ be a smooth oriented Riemannian $n$-manifold.  We use both $\vol$ and $\mu$ to denote the Riemannian volume form induced by $g$ and the given orientation. We use the Einstein summation convention throughout. We use $\sym$ to denote the second symmetric power of $T^* M$.

Given a vector bundle $E$ over $M$, we use $\Gamma(E)$ to denote the space of smooth sections of $E$. These spaces are denoted in other ways in the following cases:
\begin{itemize} \setlength\itemsep{-1mm}
\item $\Omega^k = \Gamma( \Lambda^k (T^* M))$ is the space of smooth $k$-forms on $M$;
\item $\Sym = \Gamma (\sym)$ is the space of smooth symmetric $2$-tensors on $M$.
\end{itemize}
With respect to the metric $g$ on $M$, we use $\Symo$ to denote those sections $h$ of $\Sym$ that are traceless. That is, $\Symo$ consists of those sections of $\Sym$ such that $\tr h = g^{ij} h_{ij} = 0$ in local coordinates. Then $\Sym \cong \Omega^0 \oplus \Symo$, where $h \in \Sym$ is decomposed as $h = \tfrac{1}{n} (\tr h) g + h_0$. Then we have $\Gamma(T^* M \otimes T^* M) = \Omega^0 \oplus \Symo \oplus \Omega^2$, where the splitting is pointwise orthogonal with respect to the metric on $T^* M \otimes T^* M$ induced by $g$.

{\bf Acknowledgements.} The author would like to acknowledge Jason Lotay and Naichung Conan Leung for useful discussions on the structuring of these lecture notes. The initial preparation of these notes was done while the author held a Fields Research Fellowship at the Fields Institute. The final preparation of these notes was done while the author was a visiting scholar at the Center of Mathematical Sciences and Applications at Harvard University. The author thanks both the Fields Institute and the CMSA for their hospitality.

\section{Motivation} \label{sec:motivation}

Let $(M^n, g)$ be an $n$-dimensional smooth Riemannian manifold. For all $p \in M$, we have an $n$-dimensional real vector space $T_p M$ equipped with a positive-definite inner product $g_p$, and these ``vary smoothly'' with $p \in M$. A natural question is the following:

\hspace{0.2in} \emph{What other ``natural structures'' can we put on Riemannian manifolds?}

What we would like to do is to attach such a ``natural structure'' to each tangent space $T_p M$, for all $p \in M$, in a ``smoothly varying'' way. That is, such a structure corresponds to a smooth section of some tensor bundle of $M$, satisfying some natural algebraic condition at each $p \in M$. Let us consider two examples. Let $V$ be an $n$-dimensional real vector space equipped with a positive-definite inner product $g$. Note that if we fix an isomorphism $(V, g) \cong (\R^n, \bar g)$, where $\bar g$ is the standard Euclidean inner product on $\R^n$, then the subgroup of $\GL{n, \R}$ preserving this structure is $\mathrm{O}(n)$.

\begin{ex} \label{ex:orientation}
An \emph{orientation} on $V$ is a nonzero element $\mu$ of $\Lambda^n V^*$. Let $\beta = \{ e_1, \ldots, e_n \}$ be an ordered basis of $V$. Then $e^1 \wedge \cdots \wedge e^n = \lambda \mu$ for some $\lambda \neq 0$. We say that $\beta$ is positively (respectively, negatively) oriented with respect to $\mu$ if $\lambda > 0$ (respectively, $\lambda < 0$). To demand some kind of compatibility with $g$, we can rescale $\mu$ so that $g( \mu, \mu ) = 1$. Thus $(V, g)$ admits precisely two orientations. Note that if we fix an isomorphism $(V, g) \cong (\R^n, \bar g)$, then the subgroup of $\mathrm{O}(n)$ preserving this structures is $\SO{n}$.

On a smooth manifold, an orientation is thus a nowhere-vanishing smooth section $\mu$ of $\Lambda^n (T^* M)$. That is, it is a nowhere-vanishing top form. Such a structure \emph{does not always exist}. Specifically, it exists if and only if the real line bundle $\Lambda^n (T^* M)$ is smoothly trivial. In terms of characteristic classes, this condition is equivalent to the vanishing of the first Stiefel-Whitney class $w_1 (TM)$ of the tangent bundle. (See~\cite{MS}, for example.) To demand compatibility with $g$, we can rescale $\mu$ by a positive function so that $g_p (\mu_p, \mu_p) = 1$ for all $p \in M$. This normalized $\mu$ is the \emph{Riemannian volume form} associated to the metric $g$ and the chosen orientation on $M$. It is given locally in terms of a positively oriented orthonormal frame $\{ e_1, \ldots, e_n \}$ of $TM$ by $\mu = e^1 \wedge \cdots \wedge e^n$.

An orientation compatible with the metric is called a $\SO{n}$-structure on $M$. It is equivalent to a reduction of the structure group of the frame bundle of $TM$ from $\GL{n, \R}$ to $\SO{n}$.
\end{ex}

\begin{ex} \label{ex:hermitian}
A  \emph{Hermitian structure} on $(V, g)$ is an orthogonal complex structure $J$. That is, $J$ is a linear endomorphism of $V$ such that $J^2 = - I$ and $g(Jv, Jw) = g(v, w)$ for all $v,w \in V$. It is well-known and easy to check that such a structure exists on $V$ if and only if $n = 2m$ is even. Such a structure allows us to identify the $2m$-dimensional real vector space $V$ with a $m$-dimensional complex vector space, where the linear endomorphism $J$ corresponds to multiplication by $\sqrt{-1}$. Note that if we fix an isomorphism $(V, g) \cong (\R^{2m}, \bar g)$, then the subgroup of $\mathrm{O}(n)$ preserving this structures is $\U{m} = \SO{2m} \cap \GL{m, \C}$.

On a Riemannian manifold $(M, g)$, a Hermitian structure is a smooth section $J$ of the tensor bundle $TM \otimes T^* M = \End (TM)$ such that $J^2 = - I$ (which is called an \emph{almost complex structure}) and such that $g_p (J_p X_p, J_p Y_p) = g_p (X_p, Y_p)$ for all $X_p, Y_p \in T_p M$ (which makes it \emph{orthogonal}). As in Example~\ref{ex:orientation}, such a structure does not always exist, even if $n = \dim M = 2m$ is even. There are \emph{topological obstructions} to the existence of an almost complex structure, which is equivalent to a reduction of the structure group of the frame bundle of $TM$ from $\GL{2m, \R}$ to $\GL{m, \C}$. See Massey~\cite{Massey} for discussion on this question.

Further demanding that $J$ be compatible with the metric $g$ (that is, orthogonal) is a reduction of the structure group of the frame bundle of $TM$ from $\GL{2m, \R}$ to $\U{m}$. For this reason a Hermitian structure on $M^{2m}$ is sometimes also called a \emph{$\U{m}$-structure}. Readers can consult~\cite{dKS} for a comprehensive treatment of the geometry of general $\U{m}$-structures.
\end{ex}

Again, let $V$ be an $n$-dimensional real vector space equipped with a positive-definite inner product $g$. A \emph{$\G$-structure} is a special algebraic structure we can consider on $(V, g)$ only when $n=7$. In this case, if we fix an isomorphism $(V, g) \cong (\R^7, \bar g)$, then $\G$ is the subgroup of $\SO{7}$ preserving this special algebraic structure. In order to describe this structure at the level of linear algebra, we first need to discuss the \emph{octonions}, which we do in Section~\ref{sec:octonions}. Then $\G$-structures are defined and studied in Section~\ref{sec:G2}. For the purposes of this motivational section, all the reader needs to know is that a ``$\G$-structure'' corresponds to a special kind of $3$-form $\ph$ on $M^7$.

Suppose we have a ``natural structure'' on a Riemannian manifold $(M^n, g)$, such as that of Examples~\ref{ex:orientation} or~\ref{ex:hermitian} or the mysterious $\G$-structure that is the subject of the present notes. Since we have a Riemannian metric $g$, we have a Levi-Civita connection $\nabla$ and we can further ask for the ``natural structure'' to be \emph{parallel} or \emph{covariantly constant} with respect to $\nabla$. For example:

\begin{itemize}
\item If $\mu$ is an orientation (Riemannian volume form) on $(M^n, g)$, then it is \emph{always} parallel.
\item If $J$ is an orthogonal almost complex structure on $(M^{2m}, g)$, then if we have $\nabla J = 0$, we say that $(M, g, J)$ is a \emph{K\"ahler} manifold. Such manifolds have been classically well-studied.
\item If $\ph$ is a $\G$-structure on $(M^7, g)$, then if we have $\nabla \ph = 0$, we say that $(M, g, \ph)$ is a \emph{torsion-free $\G$ manifold}. Such manifolds are discussed in Section~\ref{sec:TF} of the present notes.
\end{itemize}

\section{Algebraic structures from the octonions} \label{sec:octonions}

In this section we give an introduction to the algebra of the \emph{octonions} $\Oc$, an $8$-dimensional real normed division algebra, and to the induced algebraic structure on $\imag \Oc$, the $7$-dimensional space of \emph{imaginary octonions}. We do this by discussing both normed division algebras and spaces equipped with a cross product, and then relating the two concepts. This is not strictly necessary if the intent is to simply consider $\G$-structures, but it has the pedagogical benefit of putting both $\G$ and $\Spin{7}$ geometry into the proper wider context of geometries associated to real normed division algebras. (See Leung~\cite{Leung} for more on this perspective.)

We do not discuss all of the details here, but we do prove many of the important simple results. More details on the algebraic structure of the octonions can be found in Harvey~\cite{Harvey}, Harvey--Lawson~\cite{HL}, and Salamon--Walpuski~\cite{SW}, for example.

\subsection{Normed division algebras} \label{sec:nda}

Let $\A = \R^n$ be equipped with the standard Euclidean inner product $\langle \cdot, \cdot \rangle$.

\begin{defn} \label{defn:nda}
We say that $\A$ is a \emph{normed division algebra} if $\A$ has the structure of a (\emph{not necessarily associative!}) algebra over $\R$ with multiplicative identity $1 \neq 0$ such that
\begin{equation} \label{eq:nda}
\| a b \| = \| a \| \, \| b \| \qquad \text{ for all $a, b \in \A$},
\end{equation}
where $\| a \|^2 = \langle a, a \rangle$ is the usual Euclidean norm on $\R^n$ induced from $\langle \cdot, \cdot \rangle$. Equation~\eqref{eq:nda} imposes a \emph{compatibility condition} between the inner product and the algebra structure on $\A$.
\end{defn}
\begin{rmk} \label{rmk:nda-generality}
This is not the most general definition possible, but it suffices for our purposes. See~\cite[Appendices IV.A and IV.B]{HL} for more details.
\end{rmk}

We discuss examples of normed division algebras later in this section, although it is clear that the standard algebraic structures on $\R$ and $\C \cong \R^2$ give examples. We now define some additional structures and investigate some properties of normed division algebras. It is truly remarkable how many far reaching consequences arise solely from the fundamental identity~\eqref{eq:nda}.

\begin{defn} \label{defn:nda-stuff}
Let $\A$ be a normed division algebra. Define the \emph{real part} of $\A$, denoted $\real \A$, to be the span over $\R$ of the multiplicative identity $1 \in \A$. That is, $\real \A = \{ t 1 : t \in \R \}$. Define the \emph{imaginary part} of $\A$, denoted $\imag \A$, to be the orthogonal complement of $\real \A$ with respect to the Euclidean inner product on $\A = \R^n$. That is, $\imag \A = (\real \A)^{\perp} \cong \R^{n-1}$. Given $a \in \A$, there exists a unique decomposition
\begin{equation*}
a = \real a + \imag a, \qquad \text{ where $\real a \in \real \A$ and $\imag a \in \imag \A$}.
\end{equation*}
We define the \emph{conjugate} $\overline a$ of $a$ to be the element
\begin{equation*}
\overline a = \real a - \imag a. \qedhere
\end{equation*}
\end{defn}

Note that the map $a \mapsto \overline a$ is a linear involution of $\A$, and is precisely the isometry that is minus the reflection across the hyperplane $\imag \A$ of $\A$. It is clear that
\begin{equation} \label{eq:real-imag-proj}
\real a = \tfrac{1}{2} (a + \overline a) \qquad \text{ and } \qquad \imag a = \tfrac{1}{2} (a - \overline a).
\end{equation}
As a result, we deduce that
\begin{equation} \label{eq:imag}
\overline a = -a \quad \text{ if and only if } \quad a \in \imag \A.
\end{equation}

We now derive a slew of important identities that are all consequences of the defining property~\eqref{eq:nda}.

\begin{lemma} \label{lemma:nda-identities}
Let $a, b, c \in \A$. Then we have
\begin{equation} \label{eq:nda-identities1}
\langle ac, bc \rangle = \langle ca, cb \rangle = \langle a, b \rangle \|c\|^2,
\end{equation}
and
\begin{equation} \label{eq:nda-identities2}
\begin{aligned}
\langle a, bc \rangle & = \langle a \ol{c}, b \rangle, & \qquad \langle a, cb \rangle & = \langle \ol{c} a, b \rangle.
\end{aligned}
\end{equation}
Moreover, we also have
\begin{equation} \label{eq:conjugate}
\ol{ab} = \overline b \overline a.
\end{equation}
\end{lemma}
\begin{proof}
First observe that
\begin{align*}
\| (a + b) c \|^2 & = \| ac + bc \|^2 = \|ac\|^2 + 2 \langle ac, bc \rangle + \|bc\|^2, \\
\| a + b \|^2 \, \| c\|^2 & = (\|a\|^2 + 2 \langle a, b \rangle + \|b\|^2) \|c\|^2.
\end{align*}
Equating the left hand sides above using the fundamental identity~\eqref{eq:nda}, and again using~\eqref{eq:nda} to cancel the corresponding first and third terms on the right hand sides, we obtain
\begin{equation} \label{eq:nda-polarizeright}
\langle ac, bc \rangle = \langle a, b \rangle \|c\|^2.
\end{equation}
Similarly we can show
\begin{equation} \label{eq:nda-polarizeleft}
\langle ca, cb \rangle = \langle a, b \rangle \|c\|^2.
\end{equation}
Thus we have established~\eqref{eq:nda-identities1}. Consider the first equation in~\eqref{eq:nda-identities2}. It is clearly satisfied when $c$ is real, since then $\ol{c} = c$ and the inner product $\langle \cdot, \cdot \rangle$ is bilinear over $\R$. Because both sides of the equation are linear in $c$, it is enough to consider the situation when $c$ is purely imaginary, in which case $\ol{c} = -c$. Then $c$ is orthogonal to $1$, so $\| 1 + c \|^2 = 1 + \|c\|^2$. Applying~\eqref{eq:nda} and~\eqref{eq:nda-identities1}, we find
\begin{align*}
\langle a, b \rangle (1 + \|c\|^2) & = \langle a, b \rangle \| 1 + c \|^2 = \langle a(1+c), b(1+c) \rangle \\
& = \langle a + ac, b + bc \rangle = \langle a, b \rangle + \langle ac, bc \rangle + \langle a, bc \rangle + \langle ac, b \rangle \\
& = \langle a, b \rangle + \langle a, b \rangle \|c\|^2 + \langle a, bc \rangle + \langle ac, b \rangle.
\end{align*}
Thus we have $\langle a, bc \rangle = - \langle ac, b \rangle = \langle a \ol{c}, b \rangle$. This establishes the first equation in~\eqref{eq:nda-identities2}. The other is proved similarly. Using~\eqref{eq:nda-identities2} and the fact that conjugation is an isometry, we have
\begin{equation*}
\langle \ol{ab}, c \rangle = \langle ab, \ol{c} \rangle = \langle b, \ol{a} \, \ol{c} \rangle = \langle bc, \ol{a} \rangle = \langle c, \ol{b} \ol{a} \rangle.
\end{equation*}
Since this holds for all $c \in \A$, we deduce that $\ol{ab} = \ol{b} \ol{a}$.
\end{proof}

Lemma~\ref{lemma:nda-identities} has several important corollaries.

\begin{cor} \label{cor:nda-identities-polarized}
Let $a, b, c \in \A$. Then we have
\begin{align} \label{eq:nda-identitiesp1}
a(\ol{b} c) + b(\ol{a} c) & = 2 \langle a, b \rangle c, \\ \label{eq:nda-identitiesp2}
(a \ol{b}) c + (a \ol{c}) b & = 2 \langle b, c \rangle a, \\ \label{eq:nda-identitiesp3}
a \ol{b} + b \ol{a} & = 2 \langle a, b \rangle 1.
\end{align}
\end{cor}
\begin{proof}
Polarizing~\eqref{eq:nda-identities1}, we have
\begin{align*}
\langle a, b \rangle \|c+d\|^2 & = \langle a(c+d), b(c+d) \rangle \\
\langle a, b \rangle ( \|c\|^2 + 2 \langle c, d \rangle + \|d\|^2) & = \langle ac, bc \rangle + \langle ad, bc \rangle + \langle ac, bd \rangle + \langle ad, bd \rangle,
\end{align*}
and hence upon using~\eqref{eq:nda-identities1} to cancel the corresponding first and last terms on each side, we get
\begin{equation} \label{eq:nda-p-temp}
\langle ad, bc \rangle + \langle ac, bd \rangle = 2 \langle a, b \rangle \langle c, d \rangle.
\end{equation}
Using~\eqref{eq:nda-identities2}, we can write the above as
\begin{equation*}
\langle d, \ol{a}(bc) \rangle + \langle \ol{b}(ac), d \rangle = 2 \langle a, b \rangle \langle c, d \rangle.
\end{equation*}
Since the above holds for any $d \in \A$, we deduce that
\begin{equation*}
\ol{a}(bc) + \ol{b}(ac) = 2 \langle a, b \rangle c.
\end{equation*}
Replacing $a \mapsto \ol{a}$ and $b \mapsto \ol{b}$ and using the fact that conjugation is an isometry, we obtain~\eqref{eq:nda-identitiesp1}. Equation~\eqref{eq:nda-identitiesp2} is obtained similarly. Alternatively, one can take the conjugate of~\eqref{eq:nda-identitiesp1} and use the relation~\eqref{eq:conjugate}. Finally,~\eqref{eq:nda-identitiesp3} is the special case of~\eqref{eq:nda-identitiesp1} when $c=1$.
\end{proof}

\begin{cor} \label{cor:nda-identities-polarized-i}
Let $a, b, c \in \imag \A$. Then we have
\begin{align} \label{eq:nda-identitiesp1-i}
a(b c) + b(a c) & = -2 \langle a, b \rangle c, \\ \label{eq:nda-identitiesp2-i}
(a b) c + (a c) b & = -2 \langle b, c \rangle a, \\ \label{eq:nda-identitiesp3-i}
a b + b a & = -2 \langle a, b \rangle 1.
\end{align}
\end{cor}
\begin{proof}
These are immediate from Corollary~\ref{cor:nda-identities-polarized} and equation~\eqref{eq:imag}.
\end{proof}

\begin{cor} \label{cor:ndametric}
Let $a, b \in \A$. Then we have
\begin{equation} \label{eq:ndametric}
\langle a, b \rangle = \real (a\overline b) = \real (b \overline a) = \real (\overline b a) = \real (\overline a b)
\end{equation}
and
\begin{equation} \label{eq:ndanorm}
\| a \|^2 = a \ol{a} = \ol{a} a.
\end{equation}
\end{cor}
\begin{proof}
Using~\eqref{eq:nda-identities2}, we have $\langle a, b \rangle = \langle a, b1 \rangle = \langle a \ol{b}, 1 \rangle = \real(a \ol{b})$. The remaining equalities in~\eqref{eq:ndametric} follow from the symmetry of $\langle \cdot, \cdot \rangle$ and the fact that conjugation is an isometry. From~\eqref{eq:conjugate}, we find $\ol{\ol{a} a} = \ol{a} \, \ol{\ol{a}} = \ol{a} a$, so $\ol{a} a$ is real. Equation~\eqref{eq:ndanorm} thus follows from~\eqref{eq:ndametric}.
\end{proof}

\begin{cor} \label{cor:real-imag}
Let $a \in \A$. Then $a^2 = aa$ is real if and only if $a$ is either real or imaginary.
\end{cor}
\begin{proof}
Write $a = \real a + \imag a$. Since $\ol{\imag a} = - \imag{a}$, from~\eqref{eq:ndanorm} we have $(\imag a)^2 = - (\imag a)(\ol{\imag a}) = - \| \imag a \|^2$. Thus we have
\begin{equation*}
a^2 = (\real a + \imag a) (\real a + \imag a) = \big( (\real a)^2 - \| \imag a \|^2 \big) 1 + 2 (\real a)(\imag a).
\end{equation*}
Since the first term on the right hand side above is always real and the second term is always imaginary, we conclude that $a^2$ is real if and only if $(\real a)(\imag a) = 0$, which means that either $\real a = 0$ or $\imag a = 0$.
\end{proof}

\begin{cor} \label{cor:nda-identities}
Let $a, c \in \A$. Then we have
\begin{equation} \label{eq:nda-identities3}
\begin{aligned}
(ac) \ol{c} & = a(c\ol{c}) = \|c\|^2 a = a(\ol{c} c) = (a\ol{c}) c, \\
a(\ol{a}c) & = (a\ol{a})c = \|a\|^2 c = (\ol{a} a) c = \ol{a} (ac).
\end{aligned}
\end{equation}
\end{cor}
\begin{proof}
Observe from~\eqref{eq:nda-identities2} and~\eqref{eq:nda-identities1} that
\begin{equation*}
\langle (ac)\ol{c}, b \rangle = \langle ac, bc \rangle = \langle a, b \rangle \|c\|^2 = \langle a \|c\|^2, b \rangle = \langle a(c\ol{c}), b \rangle. 
\end{equation*}
Since this holds for all $b \in \A$, we deduce that
\begin{equation*}
(ac) \ol{c} = a(c\ol{c}).
\end{equation*}
The rest of the first identity in~\eqref{eq:nda-identities3} follows by interchanging $c$ and $\ol{c}$. The second identity in~\eqref{eq:nda-identities3} is proved similarly.
\end{proof}

We now introduce two fundamental $\A$-valued multilinear maps on $\A$.

\begin{defn} \label{defn:commutator-associator}
Let $\A$ be a normed division algebra. Define a bilinear map $[ \cdot, \cdot ] : \A^2 \to \A$ by
\begin{equation} \label{eq:commutator}
[a, b] = ab - ba \qquad \text{ for all $a, b \in \A$}.
\end{equation}
The map $[ \cdot, \cdot ]$ is called the \emph{commutator} of $\A$.

Define a trilinear map $[ \cdot, \cdot, \cdot ] : \A^3 \to \A$ by
\begin{equation} \label{eq:associator}
[a, b, c] = (ab) c - a (bc) \qquad \text{ for all $a, b, c \in \A$}.
\end{equation}
The map $[ \cdot, \cdot, \cdot ]$ is called the \emph{associator} of $\A$.
\end{defn}

It is clear that the commutator vanishes identically on $\A$ if and only if $\A$ is commutative, and similarly the associator vanishes identically on $\A$ if and only if $\A$ is associative.

\begin{prop} \label{prop:comm-assoc-alt}
The commutator and associator are both \emph{alternating}. That is, they are totally skew-symmetric in their arguments.
\end{prop}
\begin{proof}
The commutator is clearly alternating. Because $\A$ is an algebra over $\R$, the associator clearly vanishes if any of the arguments are purely real. Thus, because the associator is trilinear it suffices to show that $\A$ is alternating on $(\imag \A)^3$. If $a, b \in \imag \A$, then $\ol{a} = -a$ and $\ol{b} = -b$. Thus by~\eqref{eq:nda-identities3} we find that
\begin{align*}
- [ a, a, b] & = [ a, \ol{a}, b ] = (a \ol{a}) b - a (\ol{a} b) = 0.
\end{align*}
Similarly we have $- [a, b, b] = [a, \ol{b}, b] = (a \ol{b})b - a (\ol{b} b) = 0$. Thus $[ \cdot, \cdot, \cdot]$ is alternating in its first two arguments and in its last two arguments. Thus $[a, b, a] = - [a, a, b] = 0$ as well.
\end{proof}

The next result says that both the commutator and the associator restrict to vector-valued alternating multilinear forms on $\imag \A$.
\begin{lemma} \label{lemma:associator-imag}
If $a, b, c \in \imag \A$, then $[a,b] \in \imag \A$ and $[a, b, c] \in \imag \A$.
\end{lemma}
\begin{proof}
We need to show that $[a, b]$ and $[a,b,c]$ are orthogonal to $1$. Using the fact that $\ol{a} = - a$ for any $a \in \imag \A$, and the identities~\eqref{eq:nda-identities2} and~\eqref{eq:3-form-alt}, we compute
\begin{align*}
\langle [a,b], 1 \rangle & = \langle ab - ba, 1 \rangle = \langle b, \ol{a} \rangle - \langle a, \ol{b} \rangle \\
& = - \langle b, a \rangle + \langle a, b \rangle = 0.
\end{align*}
Similarly, noting that $\ol{bc} = \ol{c} \ol{b} = (-c)(-b) = cb$, we have
\begin{align*}
\langle [a, b, c], 1 \rangle & = \langle (ab)c - a(bc), 1 \rangle = \langle ab, \ol{c} \rangle - \langle bc, \ol{a} \rangle \\
& = -\langle ab, c \rangle + \langle bc, a \rangle = - \langle a, c\ol{b} \rangle + \langle bc, a \rangle \\
& = \langle a, cb + bc \rangle = \langle a, bc + \ol{bc} \rangle = 2 \langle a, \real (bc) \rangle = 0,
\end{align*}
as claimed.
\end{proof}

\begin{prop} \label{prop:skew-stuff}
Let $a, b, c, d \in \A$. The multilinear expressions $\langle a, [b, c] \rangle$ and $\langle a, [b,c,d] \rangle$ are both totally skew-symmetric in their arguments.
\end{prop}
\begin{proof}
The commutator and the associator are alternating by Proposition~\ref{prop:comm-assoc-alt}. Thus we need only show that $\langle a, [a,b] \rangle = 0$ and $\langle a, [a, b, c] \rangle = 0$. Using the identity~\eqref{eq:nda-identities1} we compute
\begin{equation*}
\langle a, [a, b] \rangle = \langle a, ab - ba \rangle = \|a\|^2 \langle 1, b \rangle - \|a\|^2 \langle 1, b \rangle = 0, 
\end{equation*}
and similarly using~\eqref{eq:nda-identities1} and~\eqref{eq:nda-identities2} we have
\begin{align*}
\langle a, [a, b, c] \rangle & = \langle a, (ab)c - a(bc) \rangle = \langle a \ol{c}, ab \rangle - \|a\|^2 \langle 1, bc \rangle \\
& = \|a\|^2 \langle \ol{c}, b \rangle - \|a\|^2 \langle \ol{c}, b \rangle = 0
\end{align*}
as claimed.
\end{proof}

\subsection{Induced algebraic structure on $\imag \A$} \label{sec:vcp}

Let $\A \cong \R^n$ be a normed division algebra with imaginary part $\imag \A \cong \R^{n-1}$. We define several objects on $\imag \A$ induced from the algebra structure on $\A$. Motivated by Lemma~\ref{lemma:associator-imag} and Proposition~\ref{prop:skew-stuff} the following definition is natural. The factor of $\tfrac{1}{2}$ is for convenience, to avoid factors of $2$ in equations~\eqref{eq:3-form-cross} and~\eqref{eq:iterated-cross-alt}.

\begin{defn} \label{defn:3-4-form}
Recall the statement of Proposition~\ref{prop:skew-stuff}. Define a $3$-form $\ph$ and a $4$-form $\ps$ on $\imag \A$ as follows:
\begin{align} \label{eq:3-form}
\ph(a,b,c) & = \tfrac{1}{2} \langle a, [b, c] \rangle = \tfrac{1}{2} \langle [a, b], c \rangle \qquad \text{ for $a,b,c \in \imag \A$}, \\ \label{eq:4-form}
\ps(a,b,c,d) & = \tfrac{1}{2} \langle a, [b, c, d] \rangle = - \tfrac{1}{2} \langle [a, b, c], d \rangle \quad \text{ for $a,b,c,d \in \imag \A$}.
\end{align}
The form $\ph \in \Lambda^3 (\imag \A)^*$ is called the \emph{associative} $3$-form, and the form $\ps \in \Lambda^4 (\imag \A)^*$ is called the \emph{coassociative} $4$-form for reasons that become clear in the context of calibrated geometry~\cite{Minischool-Lotay-calib, Minischool-Leung}.
\end{defn}

The $3$-form $\ph$ is intimately related to another algebraic operation on $\imag \A$ that is fundamental in $\G$-geometry, given by the following definition.

\begin{defn} \label{defn:cross-product}
Define a bilinear map $\times : \A^2 \to \A$ by
\begin{equation} \label{eq:cross-product}
a \times b = \imag(ab) \qquad \text{ for all $a, b \in \imag \A$}.
\end{equation}
Essentially, since the product in $\A$ of two imaginary elements need not be imaginary, we project to the imaginary part to define $\times$. The bilinear map $\times$ is called the \emph{vector cross product} on $\imag \A$ induced by the algebraic structure on $\A$. 
\end{defn}

The vector cross product $\times$ has several interesting properties.

\begin{lemma} \label{lemma:cross-product}
Let $a, b \in \imag \A$. The we have
\begin{align} \label{eq:cross-skew}
a \times b & = - b \times a, \\ \label{eq:cross-product1}
\langle a \times b, a \rangle & = 0, \qquad \qquad \text{ so $(a \times b) \perp a$ and $(a \times b) \perp b$}, \\ \label{eq:cross-product2}
\real(ab) & = - \langle a, b \rangle 1.
\end{align}
\end{lemma}
\begin{proof}
Recall that $\ol{a} = -a$ and $\ol{b} = -b$. Thus from~\eqref{eq:real-imag-proj} and~\eqref{eq:conjugate}, we have
\begin{equation} \label{eq:cross-comm}
2 a \times b = ab - \ol{ab} = ab - ba = [a,b].
\end{equation}
Thus $a \times b = - b \times a$.

Since $a \in \imag \A$, equation~\eqref{eq:cross-product} show that $\langle a \times b, a \rangle = \langle \imag(ab), a \rangle = \langle ab, a \rangle$. Thus, using~\eqref{eq:nda-identities1} we get $\langle a \times b, a \rangle = \langle ab, a \rangle = \langle ab, a 1 \rangle = \|a\|^2 \langle b, 1 \rangle = 0$ because $b \in \imag \A$ is orthogonal to $1 \in \real \A$.

Since $\ol{b} = -b$, equation~\eqref{eq:ndametric} gives $\langle a, b \rangle = \real(a \ol{b}) = - \real(ab)$, which is~\eqref{eq:cross-product2}.
\end{proof}

Combining equations~\eqref{eq:cross-product2} and~\eqref{eq:cross-product} gives us that the decomposition of $ab \in \A$ into real and imaginary parts is given by
\begin{equation} \label{eq:abdecomp}
ab = - \langle a, b \rangle 1 + a \times b.
\end{equation}
It then follows from~\eqref{eq:cross-comm} and~\eqref{eq:3-form} that
\begin{equation} \label{eq:3-form-cross}
\ph(a,b,c) = \langle a \times b , c \rangle \qquad \text{ for $a,b,c \in \imag \A$}.
\end{equation}
Note that since $a \times b - ab$ is real by~\eqref{eq:abdecomp}, we can equivalently write~\eqref{eq:3-form-cross} as
\begin{equation} \label{eq:3-form-alt}
\ph(a,b,c) = \langle ab, c \rangle \qquad \text{ for $a,b,c \in \imag \A$}.
\end{equation}

\begin{lemma} \label{lemma:abc}
Let $a, b, c \in \imag \A$. Then we have
\begin{equation} \label{eq:abc}
a(bc) = -\tfrac{1}{2}[a,b,c] - \ph(a,b,c)1 - \langle b, c \rangle a + \langle a, c \rangle b - \langle a, b \rangle c.
\end{equation}
\end{lemma}
\begin{proof}
Applying the three identities in Corollary~\ref{cor:nda-identities-polarized-i} repeatedly, we compute
\begin{align*}
a(bc) & = - b(ac) - 2 \langle a, b \rangle c \\
& = - b \big( -ca - 2 \langle a, c \rangle 1 \big) - 2 \langle a, b \rangle c \\
& = b(ca) + 2 \langle a, c \rangle b - 2 \langle a, b \rangle c \\
& = -c(ba) - 2 \langle b, c \rangle a + 2 \langle a, c \rangle b - 2 \langle a, b \rangle c.
\end{align*}
Also, putting $c \mapsto c$ and $b \mapsto ab$ in~\eqref{eq:nda-identitiesp3} and using~\eqref{eq:3-form-alt} gives
\begin{equation*}
c(ba) - (ab)c = c (\ol{ab}) + (ab) \ol{c} =  2 \langle ab, c \rangle 1 = 2 \ph(a,b,c) 1.
\end{equation*}
Combining the above two expressions gives
\begin{equation*}
a(bc) = - (ab)c - 2 \ph(a,b,c) 1 - 2 \langle b, c \rangle a + 2 \langle a, c \rangle b - 2 \langle a, b \rangle c.
\end{equation*}
Using $[a,b,c] = (ab)c - a(bc)$ to eliminate $(ab)c$ above and rearranging gives~\eqref{eq:abc}.
\end{proof}

Equation~\eqref{eq:abc} is used to establish the following two corollaries.

\begin{cor} \label{cor:cross-product}
Let $a, b, c \in \imag \A$. The vector cross product $\times$ on $\imag \A$ satisfies
\begin{equation} \label{eq:cross-product3}
\| a \times b \|^2 = \| a \|^2 \, \| b \|^2 - \langle a, b \rangle^2 = \| a \wedge b \|^2,
\end{equation}
and
\begin{equation} \label{eq:cross-product4}
a \times (b \times c) = - \langle a, b \rangle c + \langle a, c \rangle b - \tfrac{1}{2} [a, b, c].
\end{equation}
\end{cor}
\begin{proof}
Let $a, b \in \imag \A$. Using~\eqref{eq:abdecomp} we have $ab = - \langle a, b \rangle 1 + a \times b$ and $ba = - \langle b, a \rangle 1 + b \times a = - \langle a, b \rangle - a \times b$. Thus we have
\begin{equation*}
\langle ab, ba \rangle = \langle - \langle a, b \rangle 1 + a \times b , - \langle a, b \rangle 1 - a \times b \rangle = \langle a, b \rangle^2 - \| a \times b \|^2.
\end{equation*}
Using the above expression and equations~\eqref{eq:cross-comm} and~\eqref{eq:nda}, we compute
\begin{align*}
4 \| a \times b \|^2 & = \langle ab - ba, ab - ba \rangle = \| ab \|^2 + \| b a \|^2 - 2 \langle ab, ba \rangle \\
& = \|a \|^2 \|b\|^2 + \|b\|^2 \|a\|^2 - 2( \langle a, b \rangle^2 - \| a \times b \|^2),
\end{align*}
which simplifies to~\eqref{eq:cross-product3}. Again using~\eqref{eq:abdecomp} twice and~\eqref{eq:3-form-cross} we compute
\begin{align*}
a \times (b \times c) & = \langle a, b \times c \rangle 1 + a(b \times c)\\
& = \ph(a,b,c) 1 + a( \langle b, c \rangle + bc ) \\
& = a(bc) + \ph(a,b,c) 1 + \langle b, c \rangle a \rangle.
\end{align*}
Substituting~\eqref{eq:abc} for $a(bc)$ above gives~\eqref{eq:cross-product4}.
\end{proof}

\begin{cor} \label{cor:calib-ineq}
Let $a,b,c,d \in \imag \A$. Then the following identity holds:
\begin{equation} \label{eq:calib-ineq}
\| \tfrac{1}{2} [a, b, c] \|^2 + \big( \ph(a,b,c) \big)^2 = \| a \wedge b \wedge c \|^2.
\end{equation}
\end{cor}
\begin{proof}
Recall from Lemma~\ref{lemma:associator-imag} and Proposition~\ref{prop:skew-stuff} that $[a,b,c]$ is purely imaginary and is orthogonal to $a$, $b$, $c$. Thus, taking the norm squared of~\eqref{eq:abc} and using the fundamental relation~\eqref{eq:nda}, we have
\begin{align*}
\|a \|^2 \|b\|^2 \|c\|^2 & = \|a\|^2 \|bc\|^2 = \|a(bc)\|^2 \\
& = \| \tfrac{1}{2} [a, b, c] \|^2 + \big( \ph(a,b,c) \big)^2 + \|a\|^2 \langle b, c \rangle^2 + \|b\|^2 \langle a, c \rangle^2 + \|c\|^2 \langle a, b \rangle^2 \\
& \qquad {} - 2 \langle b, c \rangle \langle a, c \rangle \langle a, b \rangle + 2 \langle b, c \rangle \langle a, b \rangle \langle a, c \rangle - 2 \langle a, c \rangle \langle a, b \rangle \langle b, c \rangle.
\end{align*}
This can be rearranged to yield
\begin{align*}
\| \tfrac{1}{2} [a, b, c] \|^2 + \big( \ph(a,b,c) \big)^2 & = \|a \|^2 \|b\|^2 \|c\|^2 - \|a\|^2 \langle b, c \rangle^2 - \|b\|^2 \langle a, c \rangle^2 \\
& \qquad {} - \|c\|^2 \langle a, b \rangle^2 + 2 \langle a, b \rangle \langle a, c \rangle \langle b, c \rangle,
\end{align*}
which is precisely~\eqref{eq:calib-ineq}.
\end{proof}

\begin{rmk} \label{rmk:calib-ineq}
Comparing equations~\eqref{eq:3-form} and~\eqref{eq:3-form-cross}, one is tempted from~\eqref{eq:4-form} to think of the expression $\frac{1}{2} [a,b,c]$ as some kind of $3$-fold vector cross product $P(a,b,c)$, as it is a trilinear vector valued alternating form on $\imag \A$. However, equation~\eqref{eq:calib-ineq} says that $\| a \wedge b \wedge c \|^2 - \| P(a,b,c) \|^2$ is nonzero in general, whereas~\eqref{eq:cross-product3} says that $\| a \wedge b \|^2 - \| a \times b \|^2 = 0$ always. There \emph{is} a notion of \emph{$3$-fold vector cross product} (see Remark~\ref{rmk:other-VCP} below) but $[ \cdot, \cdot, \cdot]$ on $\imag \A$ is not one of them. In fact, equation~\eqref{eq:calib-ineq} is the \emph{calibration inequality} for $\ph$. It says that $|\ph(a,b,c)| \leq 1$ whenever $a,b,c$ are orthonormal, with equality if and only if $[a,b,c]=0$. See~\cite{Minischool-Lotay-calib, Minischool-Leung} in the present volume for more about the aspects of $\G$ geometry related to \emph{calibrations}. 
\end{rmk}

Equations~\eqref{eq:cross-product1} and~\eqref{eq:cross-product3} for the vector cross product $\times$ induced from the algebraic structure on $\A$ motivate the following general definition.

\begin{defn} \label{defn:cross-product-2}
Let $\mathbb V = \R^m$, equipped with the usual Euclidean inner product. We say that $V$ has a \emph{vector cross product}, which we usually simply call a \emph{cross product}, if there exists a \emph{skew-symmetric} bilinear map $\times : \mathbb V^2 \to \mathbb V$ such that, for all $a, b, c \in \mathbb V$, we have
\begin{align} \label{eq:cross-product5}
\langle a \times b, a \rangle & = 0, \qquad \text{ so $(a \times b) \perp a$ and $(a \times b) \perp b$}, \\
\label{eq:cross-product6}
\| a \times b \|^2 & = \| a \|^2 \, \| b \|^2 - \langle a, b \rangle^2 = \| a \wedge b \|^2.
\end{align}
Note that~\eqref{eq:cross-product5} and~\eqref{eq:cross-product6} are precisely~\eqref{eq:cross-product1} and~\eqref{eq:cross-product3}, respectively.
\end{defn}

\begin{rmk} \label{rmk:cross-product}
The fact that $\times$ is skew-symmetric and bilinear is equivalent to saying that $\times$ is a \emph{linear map}
\begin{equation*}
\times : \Lambda^2 (\mathbb V) \to \mathbb V.
\end{equation*}
Then~\eqref{eq:cross-product6} says that $\times$ is \emph{length preserving on decomposable elements of $\Lambda^2 (\mathbb V)$}.
\end{rmk}

\begin{rmk} \label{rmk:other-VCP}
In Definition~\ref{defn:cross-product-2} we have really defined a special class of vector cross product, called a \emph{$2$-fold vector cross product}. A more general notion of \emph{$k$-fold vector cross product}~\cite{BG} exists. When $k=1$ such a structure is an an orthogonal complex structure. When $k=3$ such a structure is related to $\Spin{7}$-geometry. See also Lee--Leung~\cite{LL} for more details. Another extensive recent reference for general vector cross products is Cheng--Karigiannis--Madnick~\cite[Section 2]{CKM}.
\end{rmk}

We have seen that any normed division algebra $\A$ gives a vector cross product on $\mathbb V = \imag \A$. In the next section we show that we can also go the other way.

\subsection{One-to-one correspondence and classification} \label{sec:one-to-one}

We claim that the normed division algebras are in \emph{one-to-one correspondence} with the spaces admitting cross products. The correspondence is seen explicitly as follows. Let $\A$ be a normed division algebra. In Section~\ref{sec:vcp} we  showed that $\mathbb V = \imag \A$ has a cross product. Conversely, suppose $\mathbb V = \R^m$ has a cross product $\times$. Define $\A = \R \oplus \mathbb V = \R^{m+1}$, with the Euclidean inner product. That is,
\begin{equation*}
\langle (s, v), (t, w) \rangle = st + \langle v, w \rangle \qquad \text{for $s, t \in \R$ and $v, w \in \mathbb V$}.
\end{equation*}
Define a multiplication on $\A$ by
\begin{equation} \label{eq:VtoA}
(s, v)(t, w) = (st - \langle v, w \rangle, sw + tv + v \times w ).
\end{equation}
The multiplication defined in~\eqref{eq:VtoA} is clearly bilinear over $\R$, so it gives $\A$ the structure of a (not necessarily associative) algebra over $\R$. It is also clear from~\eqref{eq:VtoA} that $(1,0)$ is a multiplicative identity on $\A$. We need to check that~\eqref{eq:nda} is satisfied. We compute:
\begin{align*}
\| (s, v)(t, w) \|^2 & = (st - \langle v, w \rangle)^2 + \| sw + t v + v \times w \|^2 \\
& = s^2 t^2 - 2 st \langle v, w \rangle + \langle v, w \rangle^2 + s^2 \|w\|^2 + t^2 \|v\|^2 + \| v \times w \|^2 \\
& \qquad + 2 st \langle v, w \rangle + 2 s \langle w, v \times w \rangle + 2 t \langle v, v \times w \rangle.
\end{align*}
Using~\eqref{eq:cross-product4} and~\eqref{eq:cross-product5}, the above expression simplifies to
\begin{align*}
\| (s, v)(t, w) \|^2 & = s^2 t^2 + s^2 \|w\|^2 + t^2 \|v\|^2 + \| v \|^2 \| w \|^2 \\
& = (s^2 + \|v\|^2)(t^2 + \|w\|^2) = \|(s, v)\| \, \|(t, w)\|,
\end{align*}
verifying~\eqref{eq:nda}.

Normed division algebras were classified by Hurwitz in 1898. A proof using the \emph{Cayley--Dickson doubling construction} can be found in~\cite[Chapter 6]{Harvey} or~\cite[Appendix IV.A]{HL}. There are exactly four possibilities, up to isomorphism.

The four normed division algebras are given by the following table:
{\renewcommand{\arraystretch}{1.2}
\begin{center}
\begin{tabular}{|c|c|c|c|c|}
\hline
$n = \dim \A$ & $1$ & $2$ & $4$ & $8$ \\ \hline
Symbol & $\R$ & $\C \cong \R^2$ & $\Qu \cong \R^4$ & $\Oc \cong \R^8$ \\ \hline
Name & Real numbers & Complex numbers & Quaternions or & Octonions or \\
& & & Hamilton numbers & Cayley numbers \\ \hline
\end{tabular}
\end{center}
}
Each algebra in the above table is a \emph{subalgebra} of the next one. In particular, the multiplicative unit in all cases is the usual multiplicative identity $1 \in \R$. Moreover, as we enlarge the algebras $\R \to \C \to \Qu \to \Oc$, we lose some algebraic property at each step. From $\R$ to $\C$ we lose the \emph{natural ordering}. From $\C$ to $\Qu$ we lose \emph{commutativity}. And from $\Qu$ to $\Oc$ we lose \emph{associativity}. 

The octonions $\Oc$ are also called the \emph{exceptional} division algebra and the geometries associated to $\Oc$ are known as \emph{exceptional geometries}.

By the one-to-one correspondence between normed division algebras and spaces admitting cross products, we deduce that there exist precisely four spaces with cross product, given by the following table:
{\renewcommand{\arraystretch}{1.2}
\begin{center}
\begin{tabular}{|c|c|c|c|c|}
\hline
$m = \dim \mathbb V$ & $0$ & $1$ & $3$ & $7$ \\ \hline
Symbol & $\{ 0 \} \cong \imag \R$ & $\R \cong \imag \C$ & $\R^3 \cong \imag \Qu$ & $\R^7 \cong \imag \Oc$ \\ \hline
Cross product $\times$ & trivial & trivial & standard (Hodge star) & exceptional \\ \hline
the $3$-form & $0$ & $0$ & $\ph = \mu$ is the standard & $\ph$ is the \\
$\ph \in \Lambda^3 (\mathbb V^*)$ & & & volume form & associative $3$-form \\ \hline
the $(m-3)$-form & $0$ & $0$ & $\sta \ph = 1 \in \Lambda^0 (\mathbb V^*) \cong \R$ & $\sta \ph = \ps \in \Lambda^4 (\mathbb V^*)$ is the \\
$\sta \ph \in \Lambda^{m-3} (\mathbb V^*)$ & & & is the multiplicative unit & coassociative $4$-form \\ \hline
\end{tabular}
\end{center}
}

\begin{rmk} \label{rmk:vcp-table}
Here are some remarks concerning the above table:
\begin{enumerate}[(i)]
\item When $m = 0,1$, the cross product $\times$ is trivial because $\Lambda^2 (\mathbb V) = \{ 0 \}$ in these cases.
\item When $m = 3$ we recover the standard cross product on $\R^3$. It is well-known that the standard cross product can be obtained from quaternionic multiplication by~\eqref{eq:cross-product}, and that $\langle u \times v, w \rangle = \mu(u,v,w)$ is the volume form $\mu$ evaluated on the $3$-plane $u \wedge v \wedge w$. Equivalently, the cross product is given by the Hodge star on $\R^3$. That is, $u \times v = \sta ( u \wedge v)$. In this case $\sta \ph = \sta \mu = 1$.
\item The cross product on $\R^7$ is induced in the same way from octonionic multiplication, and is called the \emph{exceptional cross product}. In this case $\ph$ is a nontrivial $3$-form on $\R^7$, and $\sta \ph = \ps$ is a nontrivial $4$-form on $\R^7$. We discuss these in more detail in Section~\ref{sec:G2-package}. \qedhere
\end{enumerate}
\end{rmk}

\subsection{Further properties of the cross product in $\R^3$ and $\R^7$} \label{sec:further-cross}

Let us investigate some further properties of the cross product. First, note that for $\mathbb V = \R^3 \cong \imag \Qu$, equation~\eqref{eq:cross-product4} reduces to the familiar $a \times (b \times c) = - \langle a, b \rangle c + \langle a, c \rangle b$, because $\Qu$ is associative. However, for $\mathbb V = \R^7 \cong \imag \Oc$, we have
\begin{equation} \label{eq:iterated-cross-product}
a \times (b \times c) = - \langle a, b \rangle c + \langle a, c \rangle b - \tfrac{1}{2} [a, b, c]
\end{equation}
where the last term \emph{does not vanish in general}. In fact using~\eqref{eq:4-form} we can write~\eqref{eq:iterated-cross-product} as
\begin{equation} \label{eq:iterated-cross-alt}
a \times (b \times c) = - \langle a, b \rangle c + \langle a, c \rangle b + \big( \ps(a,b,c,\cdot) \big)^{\sharp}
\end{equation}
where $\alpha^{\sharp}$ is the vector in $\mathbb V$ that is metric-dual to the $1$-form $\alpha \in \mathbb V^*$ via the inner product. Explicitly, $\langle \alpha^{\sharp}, b \rangle = \alpha (b)$ for all $b \in \mathbb V$.

\begin{rmk} \label{rmk:iterated-cross}
The nontriviality of the last term in~\eqref{eq:iterated-cross-product} or~\eqref{eq:iterated-cross-alt} is equivalent to the nonassociativity of $\Oc$ and is the source of the inherent \emph{nonlinearity} in geometries defined using the octonions. See also Remark~\ref{rmk:identities-nonlinear} below.
\end{rmk}

We obtain a number of important consequences from the fundamental identity~\eqref{eq:iterated-cross-product}. The remaining results in this section hold for both the cases $\mathbb V = \R^3 \cong \imag \Qu$ and $\mathbb V = \R^7 \cong \imag \Oc$, with the understanding that the associator term vanishes in the $\R^3$ case.

\begin{cor} \label{cor:iterated-cross-1}
Let $a, c \in \mathbb V$. Then we have
\begin{equation} \label{eq:iterated-cross-product-2}
a \times (a \times c) = - \| a \|^2 c + \langle a, c \rangle a.
\end{equation}
\end{cor}
\begin{proof}
Let $a = b$ in~\eqref{eq:iterated-cross-product}. The associator term vanishes by Proposition~\ref{prop:comm-assoc-alt}.
\end{proof}

\begin{rmk} \label{rmk:induced-complex-structure}
From Corollary~\ref{cor:iterated-cross-1} we deduce the following observation. Let $a \in \mathbb V$ satisfy $\| a \| = 1$. Consider the codimension one subspace $\mathbb U = (\spa\{a\})^{\perp}$ orthogonal to $a$. Since $a \times c$ is orthogonal to $c$ for all $c$, the linear map $J_a : \mathbb V \to \mathbb V$ given by $J_a (c) = a \times c$ leaves $\mathbb U$ invariant, and by~\eqref{eq:iterated-cross-product-2} we have $(J_a)^2 = - I$ on $\mathbb U$, so $J_a$ is a complex structure on $\mathbb U$.
\end{rmk}

\begin{cor} \label{cor:iterated-cross-2}
Let $a, b, c \in \mathbb V$ be orthonormal, with $a \times b = c$. Then $b \times c = a$ and $c \times a = b$.
\end{cor}
\begin{proof}
Take the cross product of $a \times b = c$ on both sides with $a$ or $b$ and use~\eqref{eq:iterated-cross-product-2}.
\end{proof}

\begin{cor} \label{cor:cross-inner-product}
Let $a,b,c,d \in \mathbb V$. Recall that
\begin{equation*}
\langle a \wedge b, c \wedge d \rangle = \det \begin{pmatrix} \langle a, c \rangle & \langle a, d \rangle \\ \langle b, c \rangle & \langle b, d \rangle \end{pmatrix} = \langle a, c \rangle \langle b, d \rangle - \langle a, d \rangle \langle b, c \rangle.
\end{equation*}
Then we have
\begin{align} \label{eq:cross-inner-product-1}
\langle a \times b, c \times d \rangle & = \langle a \wedge b, c \wedge d \rangle - \tfrac{1}{2} \langle a, [b, c, d] \rangle, \\
\label{eq:cross-inner-product-2}
\langle a \times b, a \times d \rangle & = \langle a \wedge b, a \wedge d \rangle = \|a\|^2 \langle b, d \rangle - \langle a, b \rangle \langle a, d \rangle.
\end{align}
\end{cor}
\begin{proof}
Equation~\eqref{eq:cross-inner-product-2} follows from~\eqref{eq:cross-inner-product-1} by setting $c = a$ and using Proposition~\ref{prop:skew-stuff}. To establish~\eqref{eq:cross-inner-product-1}, we compute using~\eqref{eq:3-form-cross} and the skew-symmetry of $\ph$ as follows:
\begin{equation*}
\langle a \times b, c \times d \rangle = \ph(a, b, c \times d) = - \ph( a, c \times d, b) = - \langle a \times (c \times d), b \rangle.
\end{equation*}
Using~\eqref{eq:iterated-cross-product}, the above expression becomes
\begin{align*}
\langle a \times b, c \times d \rangle & = - \langle - \langle a, c \rangle d + \langle a, d \rangle c - \tfrac{1}{2} [a, c, d], b \rangle \\
& = \langle a, c \rangle \langle b, d \rangle - \langle a, d \rangle \langle b, c \rangle + \tfrac{1}{2} \langle b, [a,c,d] \rangle.
\end{align*}
Using Proposition~\ref{prop:skew-stuff}, the above expression equals~\eqref{eq:cross-inner-product-1}.
\end{proof}

\begin{rmk} \label{rmk:cross-inner-product}
Using~\eqref{eq:4-form}, when $n=7$ we can also write~\eqref{eq:cross-inner-product-1} as
\begin{equation} \label{eq:fund-relation-2}
\langle a \times b, c \times d \rangle = \langle a \wedge b, c \wedge d \rangle - \ps (a, b, c, d).
\end{equation}
Recall from Remark~\ref{rmk:iterated-cross} that the nontriviality of $\ps$ is equivalent to the nonassociativity of $\Oc$. Thus the above equation says that the nonassociativity of $\Oc$ is also equivalent to the fact that
\begin{equation*}
\langle a \times b, c \times d \rangle \neq \langle a \wedge b, c \wedge d \rangle
\end{equation*}
in general.

By contrast, when $n=3$ the associator vanishes, and we do have $\langle a \times b, c \times d \rangle = \langle a \wedge b, c \wedge d \rangle$ in this case. This corresponds, by (ii) of Remark~\ref{rmk:vcp-table}, to the fact that $a \times b = \sta( a \wedge b)$ and $\sta$ is an isometry.
\end{rmk}

\section{The geometry of $\G$-structures} \label{sec:G2}

In this section we discuss $\G$-structures, first on $\R^7$ and then on smooth $7$-manifolds, including a discussion of the decomposition of the space of differential forms and of the torsion of a $\G$-structure.

\subsection{The canonical $\G$-structure on $\R^7$} \label{sec:G2-package}

In this section we describe in more detail the canonical $\G$-structure on $\R^7 \cong \imag \Oc$. This standard ``$\G$-package'' on $\R^7$ consists of the standard Euclidean metric $g_{\oo}$, for which the standard basis $e_1, \ldots, e_7$ is orthonormal, the standard volume form $\mu_{\oo} = e^1 \wedge \cdots \wedge e^7$ associated to $g_{\oo}$ and the standard orientation, the ``associative'' $3$-form $\ph_{\oo}$, the ``coassociative'' $4$-form $\ps_{\oo}$, and finally the ``cross product'' $\times_{\oo}$ operation. We use the ``$\oo$'' subscript for the standard $\G$-package $(g_{\oo}, \mu_{\oo}, \ph_{\oo}, \ps_{\oo}, \times_{\oo})$ on $\R^7$ to distinguish it from a general $\G$-structure on a smooth $7$-manifold which is defined in Section~\ref{sec:G2-on-manifolds}. We  also use $\| \cdot \|_{\oo}$ to denote both the norm on $\R^7$ induced from the inner product $g_{\oo}$ and also the induced norm on $\Lambda^{\bu} (\R^7)^*$.

We identify $\R^7 \cong \imag \Oc$. Recall from Definition~\ref{defn:3-4-form} that the associative $3$-form $\ph_{\oo}$ and the coassociative $4$-form $\ps_{\oo}$ are given by
\begin{align*}
\ph_{\oo}(a,b,c) & = \tfrac{1}{2} \langle [a, b], c \rangle \qquad \text{ for $a,b,c \in \R^7$}, \\
\ps_{\oo}(a,b,c,d) & = - \tfrac{1}{2} \langle [a, b, c], d \rangle \quad \text{ for $a,b,c,d \in \R^7$}.
\end{align*}
Using the octonion multiplication table, one can show that with respect to the standard dual basis $e^1, \ldots, e^7$ on $(\R^7)^*$, and writing $e^{ijk} = e^i \wedge e^j \wedge e^k$ and similarly for decomposable $4$-forms, we have
\begin{equation} \label{eq:ph-ps-coords}
\begin{aligned}
\ph_{\oo} & = e^{123} - e^{167} - e^{527} - e^{563} - e^{415} - e^{426} - e^{437}, \\
\ps_{\oo} & = e^{4567} - e^{4523} - e^{4163} - e^{4127} - e^{2637} - e^{1537} - e^{1526}.
\end{aligned}
\end{equation}
It is immediate that
\begin{equation*}
\ps_{\oo} = \st_{\oo} \ph_{\oo},
\end{equation*}
where $\st_{\oo}$ is the Hodge star operator induced from $(g_{\oo}, \mu_{\oo})$. The explicit expressions for $\ph_{\oo}$ and $\ps_{\oo} = \st_{\oo} \ph_{\oo}$ in~\eqref{eq:ph-ps-coords} are not enlightening and need not be memorized by the reader. There is a particular method to the seeming madness in which we have written $\ph_{\oo}$ and $\ps_{\oo}$, which is explained in~\cite{K-notes} in relation to the standard $\SU{3}$-structure on $\R^7 = \C^3 \oplus \R$, where $z^1 = x^1 + i x^5$, $z^2 = x^2 + i x^6$, $z^3 = x^3 + i x^7$ are the complex coordinates on $\C^3$ and $x^4$ is the coordinate on $\R$.

One piece of information to retain from~\eqref{eq:ph-ps-coords} is that
\begin{equation} \label{eq:ph-ps-7}
\| \ph_{\oo} \|_{\oo}^2 = \| \ps_{\oo} \|_{\oo}^2 = 7,
\end{equation}
which is equivalent to the identity $\ph_{\oo} \wedge \ps_{\oo} = 7 \mu_{\oo}$. (These facts are analogous to the identities $\| \omega_{\oo} \|_{\oo}^2 = 2m$ and $\tfrac{1}{m!} \omega_{\oo}^m = \mu_{\oo}$ for the standard K\"ahler form $\omega_{\oo}$ on $\C^m$ with respect to the Euclidean metric.)

We now use this standard ``$\G$-package'' on $\R^7$ to give a definition of the group $\G$.

\begin{defn} \label{defn:group-G2}
The group $\G$ is defined to be the subgroup of $\GL{7, \R}$ that preserves the standard $\G$-package on $\R^7$. That is,
\begin{equation*}
\G = \{ A \in \GL{7, \R} : A^* g_{\oo} = g_{\oo}, A^* \mu_{\oo} = \mu_{\oo}, A^* \ph_{\oo} = \ph_{\oo} \}.
\end{equation*}
Note that because $g_{\oo}$ and $\mu_{\oo}$ determine the Hodge star operator $\st_{\oo}$, which in turn from $\ph_{\oo}$ determines $\ps_{\oo}$, and because $g_{\oo}$ and $\ph_{\oo}$ together determine $\times_{\oo}$, it follows that any $A \in \G$ also preserves $\ps_{\oo}$ and $\times_{\oo}$. (But see Theorem~\ref{thm:Bryant} below.) Moreover, since by definition $A \in \G$ preserves the standard Euclidean metric and orientation on $\R^7$, we see that $\G$ as defined above is a \emph{subgroup} of $\SO{7, \R}$.
\end{defn}

\begin{thm}[Bryant~\cite{Bryant-A}] \label{thm:Bryant}
Define $K = \{ A \in \GL{7, \R} : A^* \ph_{\oo} = \ph_{\oo} \}$. Then in fact $K = \G$. That is, if $A \in \GL{7, \R}$ preserves $\ph_{\oo}$, then it also automatically preserves $g_{\oo}$ and $\mu_{\oo}$ as well.
\end{thm}
\begin{proof}
One can show using the explicit form~\eqref{eq:ph-ps-coords} for $\ph_{\oo}$ in terms of the standard basis $e^1, \ldots, e^7$ of $(\R^7)^*$ that
\begin{equation} \label{eq:metric-from-ph}
(a \hk \ph_{\oo}) \wedge (b \hk \ph_{\oo}) \wedge \ph_{\oo} = - 6 g_{\oo}(a,b) \mu_{\oo}.
\end{equation}
It follows from~\eqref{eq:metric-from-ph} that if $A^* \ph_{\oo} = \ph_{\oo}$, then
\begin{equation} \label{eq:metric-from-ph-2}
(A^* g_{\oo})(a,b) A^* \mu_{\oo} = g_{\oo} (Aa, Ab) (\det A) \mu_{\oo} = g_{\oo} (a,b) \mu_{\oo}.
\end{equation}
Thus we have $(\det A) g_{\oo}(Aa, Ab) = g_{\oo}(a,b)$, or equivalently in terms of matrices, $g_{\oo} = (\det A) A^T g_{\oo} A$. Taking determinants of both sides, and observing that these are all $7 \times 7$ matrices, gives $\det g_{\oo} = (\det A)^9 \det g_{\oo}$, so $\det A = 1$ and $A^* \mu_{\oo} = \mu_{\oo}$. But then~\eqref{eq:metric-from-ph-2} says that $A^* g_{\oo} = g_{\oo}$ as claimed.
\end{proof}

\begin{rmk} \label{rmk:Bryant-sign}
In Bryant~\cite{Bryant-A} the equation~\eqref{eq:metric-from-ph} has a $+6$ on the right hand side rather than our $-6$, because of a different orientation convention. See also Remark~\ref{rmk:signs} below.
\end{rmk}

\begin{rmk} \label{rmk:Bryant}
Theorem~\ref{thm:Bryant} appears in~\cite{Bryant-A}. Robert Bryant claims that it is a much older result, due to \'Elie Cartan. While this is almost certainly true, most mathematicians know this result as ``Bryant's Theorem'' as~\cite{Bryant-A} is the earliest accessible reference for this result that is widely known. See Agricola~\cite{A} for more about this history of the group $\G$.
\end{rmk}

\begin{cor} \label{cor:G2-AutO}
The group $\G$ can equivalently be defined as the \emph{automorphism group} $\mathrm{Aut}(\Oc)$ of the normed division algebra $\Oc$ of octonions.
\end{cor}
\begin{proof}
Let $A \in \mathrm{Aut}(\Oc)$. Since $A$ is an algebra automorphism we have $A(1) = 1$ and thus $A(t 1) = t$ for all $t \in \R$. Now suppose $p \in \imag \Oc$. Then $p^2 = - p \ol{p} = - \| p \|_{\oo}^2$ is real. Thus we have
\begin{equation*}
( A(p) )^2 = A(p)A(p) = A(p^2) = A(-\|p\|_{\oo}^2) = - \| p \|_{\oo}^2
\end{equation*}
is real. By Corollary~\ref{cor:real-imag} we deduce that $A(p)$ must be real or imaginary. Suppose it is real. Then $A(p) = t1$ for some $t \in \R$. But then $A(p) = A(t1)$ and $p \neq t1$ since $p$ is imaginary. This contradicts the invertibility of $A$. Thus $A(p)$ must be imaginary. This means $\ol{A(p)} = - A(p)$ whenever $p$ is imaginary.

Now let  $p = (\real p) 1 + (\imag p)$. Since $A$ is linear over $\R$ and $A(1) = 1$, we get $A(p) = (\real p) 1 + A(\imag p)$. But then $\ol{A(p)} = (\real p) 1 - A(\imag p) = A(\ol{p})$. It follows that
\begin{equation*}
\| A(p) \|_{\oo}^2 = A(p) \ol{A(p)} = A(p) A(\ol{p}) = A(p \ol{p}) = A(\| p \|_{\oo}^2) = \| p \|_{\oo}^2.
\end{equation*}
Thus $\|A(p)\|_{\oo} = \|p\|_{\oo}$, and from $A(1) = 1$ and $A(\imag \Oc) \subseteq (\imag \Oc)$ we conclude that $A \in \Or{7}$. Finally, from~\eqref{eq:3-form-alt}, if $a, b, c \in \imag \Oc$ we get
\begin{align*}
(A^* \ph_{\oo})(a,b,c) & = \ph_{\oo}(Aa,Ab,Ac) = \langle (Aa)(Ab), Ac \rangle \\
& = \langle A(ab), Ac \rangle = \langle ab, c \rangle = \ph_{\oo}(a,b,c).
\end{align*}
Thus $A^* \ph_{\oo} = \ph_{\oo}$, so by Theorem~\ref{thm:Bryant} we deduce that $A \in \G$.

Conversely, if $A \in \G$, then $A$ preserves the cross product and the inner product, so if we extend $A$ linearly from $\R^7 \cong \imag \Oc$ to $\Oc = \R \oplus \R^7$ by setting $A(1) = 1$, then it follows immediately from~\eqref{eq:VtoA} that $A(ab) = A(a) A(b)$ for all $a, b \in \Oc$, so $A \in \mathrm{Aut}(\Oc)$.
\end{proof}

\begin{rmk} \label{rmk:no-decouple}
Theorem~\ref{thm:Bryant} is an \emph{absolutely crucial} ingredient of $\G$-geometry. It says that the $3$-form $\ph_{\oo}$ \emph{determines} the orientation $\mu_{\oo}$ and the metric $g_{\oo}$ in a \emph{highly nonlinear way}. This is in stark contrast to the situation of the standard $\U{m}$-structure on $\C^m$, which consists of the Euclidean metric $g_{\oo}$, the standard complex structure $J_{\oo}$ on $\C^m$, and the associated K\"ahler form $\omega_{\oo}$, which are all related by
\begin{equation} \label{eq:Um-structure}
\omega_{\oo} (a, b) = g_{\oo} (J_{\oo} a, b).
\end{equation}
Moreover, the standard volume form is $\mu_{\oo} = \frac{1}{m!} \omega_{\oo}^m$. Equation~\eqref{eq:Um-structure} should be compared to~\eqref{eq:3-form-cross}. The almost complex structure $J_{\oo}$ is the analogue of the cross product $\times_{\oo}$, and the $2$-form $\omega_{\oo}$ is the analogue of the $3$-form $\ph_{\oo}$. However, in the case of the standard $\U{m}$-structure, the $2$-form $\omega_{\oo}$ \emph{does not} determine the metric. (Although it \emph{does} determine the orientation.) The correct way to think about~\eqref{eq:Um-structure} is that knowledge of any two of $g_{\oo}, J_{\oo}, \omega_{\oo}$ uniquely determines the third. This is encoded by the following Lie group relation:
\begin{equation*}
\Or{2m, \R} \cap \GL{m, \C} = \Or{2m, \R} \cap \Sp{m, \R} = \GL{m, \C} \cap \Sp{m, \R} = \U{m}.
\end{equation*}
Colloquially, we say that the intersection of any two of Riemannian, complex, and symplectic geometry is K\"ahler geometry. By constrast, $\G$ geometry \emph{does not ``decouple''} in any such way. It is \emph{not} the intersection of Riemannian geometry with any other ``independent'' geometry. The $3$-form $\ph_{\oo}$ determines \emph{everything else}.
\end{rmk}

Let us consider how we should think about the group $\G$, which by Theorem~\ref{thm:Bryant} is described as a particular \emph{subgroup} of $\SO{7, \R}$. Before we can do that, we need a preliminary result.

\begin{lemma} \label{lemma:G2-triples}
Let $f_1, f_2, f_4$ be a triple of orthonormal vectors in $\R^7$ such that $\ph_{\oo}(f_1, f_2, f_4) = 0$. Define
\begin{equation} \label{eq:G2-triple-frame}
f_3 = f_1 \times_{\oo} f_2, \qquad f_5 = f_1 \times_{\oo} f_4, \qquad f_6 = f_2 \times_{\oo} f_4, \qquad f_7 = f_3 \times_{\oo} f_4 = (f_1 \times_{\oo} f_2) \times_{\oo} f_4.
\end{equation}
Then the ordered set $\{ f_1, \ldots, f_7 \}$ is an \emph{oriented orthonormal basis} of $\R^7$.
\end{lemma}
\begin{proof}
One can check using equations~\eqref{eq:cross-product1},~\eqref{eq:3-form-cross}, and~\eqref{eq:cross-inner-product-2}, together with the hypotheses that $\{ f_1, f_2, f_4 \}$ are orthonormal and $\ph_{\oo}(f_1, f_2, f_4) = 0$, that $\langle f_i, f_j \rangle = \delta_{ij}$ for all $i,j$ so the set is orthonormal. Most of these are immediate. We demonstrate one of the less trivial cases. Using Corollary~\ref{cor:iterated-cross-2}, we deduce that $f_3 \times_{\oo} f_1 = f_2$. Thus we have
\begin{align*}
g_{\oo} ( f_1, f_7 ) & = g_{\oo} ( f_1, f_3 \times_{\oo} f_4 ) = \ph_{\oo} (f_1, f_3, f_4) \\
& = - \ph_{\oo} (f_3, f_1, f_4) = -g_{\oo} (f_3 \times_{\oo} f_1, f_4) = -g_{\oo} (f_2, f_4) = 0.
\end{align*}
It remains to show $\{ f_1, \ldots, f_7 \}$ induces the same orientation as $\{ e_1, \ldots, e_7 \}$. When $f_k = e_k$ for $k=1,2,4$, then it follows from the octonion multiplication table and~\eqref{eq:G2-triple-frame} that $f_k = e_k$ for all $k = 1, \ldots, 7$. It is then not hard to see that the matrix in $A \in \Or{7}$ given by
\begin{equation*}
A = \begin{pmatrix} f_1 \, | \, f_2 \, | \, f_3 \, | \, f_4 \, | \, f_5 \, | \, f_6 \, | \, f_7 \end{pmatrix}
\end{equation*}
can be obtained from the identity matrix by a product of three elements of $\SO{7}$. Thus $A \in \SO{7}$ and hence $\{ f_1, \ldots, f_7 \}$ is oriented.
\end{proof}

\begin{cor} \label{cor:to-see-G2}
The group $\G$ can be viewed explicitly as the subgroup of $\SO{7}$ consisting of those elements $A \in \SO{7}$ of the form
\begin{equation} \label{eq:to-see-G2}
A = \begin{pmatrix} f_1 \, | \, f_2 \, | \, f_1 \times_{\oo} f_3 \, | \, f_4 \, | \, f_1 \times_{\oo} f_4 \, | \, f_2 \times_{\oo} f_4 \, | \, (f_1 \times_{\oo} f_2) \times_{\oo} f_4 \end{pmatrix}
\end{equation}
where $\{ f_1, f_2, f_4 \}$ is an orthonormal triple satisfying $\ph_{\oo} (f_1, f_2, f_4) = 0$. (This means that the cross product of any two of $\{ f_1, f_2, f_4 \}$ is orthogonal to the third.)
\end{cor}
\begin{proof}
By Lemma~\ref{lemma:G2-triples}, every matrix of the form~\eqref{eq:to-see-G2} does lie in $\SO{7}$. By Theorem~\ref{thm:Bryant}, a matrix $A \in \SO{7}$ is in $\G$ if and only if $A$ preserves the vector cross product $\times_{\oo}$. Since $A$ takes $e_k$ to $f_k$, it follows from the fact that
\begin{equation*}
e_3 = e_1 \times_{\oo} e_2, \qquad e_5 = e_1 \times_{\oo} e_4, \qquad e_6 = e_2 \times_{\oo} e_4, \qquad e_7 = e_3 \times_{\oo} e_4 = (e_1 \times_{\oo} e_2) \times_{\oo} e_4,
\end{equation*}
that such the elements of $\G$ are precisely the matrices of the form~\eqref{eq:to-see-G2}.
\end{proof}

\begin{rmk} \label{rmk:to-see-dim}
We can argue from Corollary~\ref{cor:to-see-G2} that $\dim \G = 14$, as follows. We know $\G$ corresponds to the set of orthonormal triples $\{ f_1, f_2, f_4 \}$ such that $f_4$ is orthogonal to $f_1$, $f_2$, and $f_1 \times_{\oo} f_2$. Thus $f_1$ is any unit vector in $\R^7$, so it lies on $S^6$. Then $f_2$ must be orthogonal to $f_1$, so it lies on the unit sphere $S^5$ of the $\R^6$ that is orthogonal to $f_1$. Finally, $f_4$ must be orthogonal to $f_1$, $f_2$, and $f_1 \times_{\oo} f_2$, so it lies on the unit sphere $S^3$ of the $\R^4$ that is orthogonal to these three vectors. Thus we have $6+5+3=14$ degrees of freedom, so $\dim \G = 14$.

(In fact, $\G$ is a connected, simply-connected, compact Lie subgroup of $\SO{7}$.)
\end{rmk}

\subsection{$\G$-structures on smooth $7$-manifolds} \label{sec:G2-on-manifolds}

In this section, as discussed in Section~\ref{sec:motivation}, we equip a smooth $7$-manifold with the ``$\G$ package'' at each tangent space, in a smoothly varying way.

\begin{defn} \label{defn:G2-structure}
Let $M^7$ be a smooth $7$-manifold. A $\G$-structure on $M$ is a smooth $3$-form $\ph$ on $M$ such that, at every $p \in M$, there exists a linear isomorphism $T_p M \cong \R^7$ with respect to which $\ph_p \in \Lambda^3 (T_p^* M)$ corresponds to $\ph_{\oo} \in \Lambda^3 (\R^7)^*$. Therefore, because $\ph_{\oo}$ induces $g_{\oo}$ and $\mu_{\oo}$, a $\G$-structure $\ph$ on $M$ induces a Riemannian metric $g_{\ph}$ and associated Riemannian volume form $\mu_{\ph}$. These in turn induce a Hodge star operator $\sta_{\ph}$ and dual $4$-form $\ps = \sta_{\ph} \ph$.
\end{defn}

Thus if $\ph$ is a $\G$-structure on $M$, then at every point $p \in M$, there exists a basis $\{ e_1, \ldots, e_7 \}$ of $T_p M$ with respect to which $\ph_p = \ph_{\oo}$ from~\eqref{eq:ph-ps-coords}. Note that in general we \emph{cannot} choose a \emph{local frame} on an open set $U$ in $M$ with respect to which $\ph$ takes the standard form in~\eqref{eq:ph-ps-coords}, \emph{we can only do this at a single point}. This is analogous to how, in a manifold $(M^{2m}, g, J, \omega)$ with $\U{m}$-structure, we can always find a basis of $T_p M$ for any $p \in M$ in which the ``$\U{m}$ package'' assumes the standard form on $\C^m$, but we cannot in general do this on an open set. (See~\cite{dKS} for a comprehensive treatment of $\U{m}$-structures.)

Not every smooth $7$-manifold admits $\G$-structures. A $\G$-structure is equivalent to a reduction of the structure group of the frame bundle of $M$ from $\GL{7, \R}$ to $\G \subset \SO{7}$. As such, the existence of a $\G$-structure is entirely a topological question.

\begin{prop} \label{prop:top-G2}
A smooth $7$-manifold $M$ admits a $\G$-structure if and only if $M$ is both \emph{orientable} and \emph{spinnable}. This is equivalent to the vanishing of the first two Stiefel-Whitney classes $w_1(TM)$ and $w_2(TM)$.
\end{prop}
\begin{proof}
See Lawson--Michelsohn~\cite[Chapter IV, Theorem 10.6]{LM} for a proof.
\end{proof}

Therefore, while not all smooth $7$-manifolds admit $\G$-structures, there are many that do and they are completely characterized by Proposition~\ref{prop:top-G2}.

There is a much more concrete way to understand when a $3$-form $\ph$ on $M$ is a $\G$-structure. It can be considered as a ``working differential geometer's definition of $\G$-structure'', and is described as follows. Let $\ph \in \Omega^3 (M^7)$. Let $x^1, \ldots, x^7$ be local coordinates on an open set $U$ in $M$. For $i, j \in \{ 1, \ldots, 7 \}$, define a smooth function $B_{ij}$ on $U$ by
\begin{equation} \label{eq:Bij}
- 6 B_{ij} \, \dx{1} \wedge \cdots \wedge \dx{7} = \Big( \ddx{i} \hk \ph \Big) \wedge \Big( \ddx{j} \hk \ph \Big) \wedge \ph.
\end{equation}
Since $2$-forms commute, we have $B_{ij} = B_{ji}$. In fact, comparison with~\eqref{eq:metric-from-ph} shows that if $\ph$ is a $\G$-structure, we must have $B_{ij} = g_{ij} \sqrt{\det g}$, since $\mu = \sqrt{\det g} \dx{1} \wedge \cdots \wedge \dx{7}$ is the Riemannian volume form in local coordinates. Hence $\det B = (\sqrt{\det g})^7 \det g = (\det g)^{\frac{9}{2}}$ and thus $\sqrt{\det g} = (\det B)^{\frac{1}{9}}$. Solving for $g_{ij}$ gives
\begin{equation} \label{eq:gij}
g_{ij} = \frac{1}{(\det B)^{\frac{1}{9}}} B_{ij}.
\end{equation}
We say that $\ph \in \Omega^3 (M^7)$ is a $\G$-structure if this recipe \emph{actually works} to construct a Riemannian metric. Thus we must have both:
\begin{enumerate}[(i)]
\item $\det B$ must be nonzero everywhere on $U$,
\item $g_{ij}$ as defined in~\eqref{eq:gij} must be positive definite everywhere on $U$.
\end{enumerate}
Of course, these two conditions must hold in any local coordinates $x^1, \ldots, x^7$ on $M$. But the advantage of this way of thinking about $\G$-structures (besides it being very concrete) is that it allows us to easily see that the condition of $\ph$ being a $\G$-structure is an \emph{open condition}. That is, if $\ph$ is a $\G$-structure, and $\tilde \ph$ is another smooth $3$-form on $M$ sufficiently close to $\ph$ (in the $C^0$-norm with respect to any Riemannian metric on $M$) then $\tilde \varphi$ will also be a $\G$-structure. This is because both conditions (i) and (ii) above are open conditions at each point $p$ of $M$.

We conclude that, if the space of $\G$-structures on $M$ is \emph{nonempty}, then it can be identified with a space $\Omega^3_+$ of smooth sections of a fibre bundle $\Lambda^3_+ (T^* M)$ whose fibres are open subsets of the corresponding fibres of the bundle $\Lambda^3 (T^* M)$. The space $\Omega^3_+$ is also called the space of \emph{nondegenerate} or \emph{positive} or \emph{stable} $3$-forms on $M$.

\begin{rmk} \label{rmk:openness}
There is another way of seeing that the condition of being a $\G$-structure is open. At any point $p \in M$, the space of all $\G$-structures $\Lambda^3_+ (T_p^* M)$ can be identified with the orbit of $\ph_{\oo}$ in $\Lambda^3 (\R^7)^*$ by the action of $\GL{7, \R}$ quotiented by the stabilizer subgroup of $\ph_{\oo}$, which is $\G$ by Theorem~\ref{thm:Bryant}. Since $\dim \GL{7, \R} = 49$, and $\dim \G = 14$, we have $\dim \Lambda^3_+ (T_p^* M) = 49 - 14 = 35 = \dim \Lambda^3 (T_p^* M)$, and thus $\Lambda^3_+ (T_p^* M)$ is an open set of $\Lambda^3 (T_p^* M)$. See Hitchin~\cite{Hi} for a general discussion of \emph{stable forms}.
\end{rmk}

\begin{rmk} \label{rmk:isometric}
The nonlinear map $\ph \to g$ is not one-to-one. In fact, given a metric $g$ on $M$ induced from a $\G$-structure $\ph$, at each point $p \in M$, the space of $\G$-structures at $p$ inducing $g_p$ is diffeomorphic to $\R \PR^7$. Thus the $\G$-structures inducing the same metric $g$ correspond to sections of an $\R \PR^7$-bundle over $M$. See~\cite[page 10, Remark 4]{Bryant} for more details on \emph{isometric $\G$-structures}.
\end{rmk}

Let $(M, \ph)$ be a manifold with $\G$-structure, and let $g$ be the induced metric. Let $\ps = \sta_{\ph} \ph$ denote the dual $4$-form. The vital relation~\eqref{eq:iterated-cross-alt}, which is equivalent to~\eqref{eq:fund-relation-2} leads to fundamental local coordinate identities relating $\ph$, $\ps$, and $g$.

\begin{thm} \label{thm:G2-identities}
In local coordinates on $M$, the tensors $\ph$, $\ps$, and $g$ satisfy the following relations:
\begin{align} \label{eq:phph1}
\ph_{ijk} \ph_{abc} g^{kc} & = g_{ia} g_{jb} - g_{ib} g_{ja} - \ps_{ijab}, \\ \label{eq:phph2}
\ph_{ijk} \ph_{abc} g^{jb} g^{kc} & = 6 g_{ia}, \\ \label{eq:phps1}
\ph_{ijk} \ps_{abcd} g^{kd} & = g_{ia} \ph_{jbc} + g_{ib} \ph_{ajc} + g_{ic} \ph_{abj} - g_{aj} \ph_{ibc} - g_{bj} \ph_{aic} - g_{cj} \ph_{abi}, \\ \label{eq:phps2}
\ph_{ijk} \ps_{abcd} g^{jc} g^{kd} & = - 4 \ph_{iab}, \\ \label{eq:psps2}
\ps_{ijkl} \ps_{abcd} g^{kc} g^{ld} & = 4 g_{ia} g_{jb} - 4 g_{ib} g_{ja} - 2 \ps_{ijab}, \\ \label{eq:psps3}
\ps_{ijkl} \ps_{abcd} g^{jb} g^{kc} g^{ld} & = 24 g_{ia}.
\end{align}
\end{thm}
\begin{proof}
These are derived from the relation~\eqref{eq:iterated-cross-alt} or equivalently~\eqref{eq:fund-relation-2}. Indeed, the first identity~\eqref{eq:phph1} is precisely~\eqref{eq:fund-relation-2} expressed in local coordinates. The explicit details can be found in~\cite[Section A.3]{K-flows}.
\end{proof}

Of course, there are many other possible contractions of $\ph$, $\ps$, and $g$. In Theorem~\ref{thm:G2-identities} we only list those that show up most frequently in practice.

\begin{rmk} \label{rmk:identities-nonlinear}
The identities for $\G$-structures in Theorem~\ref{thm:G2-identities} should be contrasted with the analogue for $\U{m}$-structures. First, we have only a single form $\omega$, as opposed to the two forms $\ph$ and $\ps$. Moreover, from $\omega_{ab} = J^c_a g_{cb}$, which comes from~\eqref{eq:Um-structure}, and the fact that $J^2 = -I$, we find that $\omega_{ia} \omega_{jb} g^{ab} = g_{ij}$. This is much simpler than~\eqref{eq:phph1} as the right hand side only involves the metric $g$. This again illustrates the ``increased nonlinearity'' of $\G$ geometry, as mentioned in Remark~\ref{rmk:iterated-cross} above.
\end{rmk}

\subsection{Decomposition of $\Omega^{\bullet}$ into irreducible $\G$ representations} \label{sec:forms}

Let $(M, \ph)$ be a manifold with $\G$-structure. The bundle $\Lambda^{\bu} (T^* M) = \oplus_{k=1}^7 \Lambda^k (T^* M)$ decomposes into irreducible representations of $\G$. This in turn induces a decomposition of the space $\Omega^k = \Gamma(\Lambda^k (T^* M) )$ of smooth $k$-forms on $M$. This is entirely analogous to how, on a manifold with almost complex structure, the space $\Omega^{\bu}_{\C} = \Gamma( \Lambda^{\bu} (T^* M) \otimes \C )$ of complex-valued forms decomposes into ``forms of type $(p,q)$''.

By Theorem~\ref{thm:Bryant}, all the tensors determined by $\ph$ will be invariant under $\G$ and hence  any subspaces of $\Omega^k$ defined using $\ph$, $\ps$, $g$, and $\sta$ will be $\G$ representations. The space $\Omega^k$ is \emph{irreducible} if $k = 0,1,6,7$. However, for $k=2,3,4,5$ we have a nontrivial decomposition. Since $\Omega^k = \sta \Omega^{7-k}$, the decompositions of $\Omega^5$ and $\Omega^4$ are obtained by taking $\sta$ of the decompositions of $\Omega^2$ and $\Omega^3$, respectively.

In fact we have
\begin{align*}
\Omega^2 & = \Omega^2_7 \oplus \Omega^2_{14}, \\
\Omega^3 & = \Omega^3_1 \oplus \Omega^3_7 \oplus \Omega^3_{27},
\end{align*}
where $\Omega^k_l$ has (pointwise) dimension $l$ and these decompositions are \emph{orthogonal} with respect to $g$. These spaces are described invariantly as follows:
\begin{equation} \label{eq:Omega2}
\begin{aligned}
\Omega^2_7 & = \{ X \hk \ph \mid X \in \Gamma(TM) \} = \{ \beta \in \Omega^2 \mid \sta ( \ph \wedge \beta ) = -2 \beta \}, \\
\Omega^2_{14} & = \{ \beta \in \Omega^2 \mid \beta \wedge \ps = 0 \} = \{ \beta \in \Omega^2 \mid \sta ( \ph \wedge \beta) = \beta \},
\end{aligned}
\end{equation}
and
\begin{equation} \label{eq:Omega3}
\begin{aligned}
\Omega^3_1 & = \{ f \ph \mid f \in \Omega^0 \}, \qquad \qquad \Omega^3_7 = \{ X \hk \ps \mid X \in \Gamma(TM) \}, \\
\Omega^3_{27} & = \{ \gamma \in \Omega^3 \mid \gamma \wedge \ph = 0, \gamma \wedge \ps = 0 \}.
\end{aligned}
\end{equation}

It is sometimes necessary to get our hands dirty, so we need to describe these subspaces in terms of local coordinates. Consider first the $\G$-invariant linear map $P : \Omega^2 \to \Omega^2$ given by $P\beta = \sta (\ph \wedge \beta)$. If we write $\beta = \tfrac{1}{2} \beta_{ij} \dx{i} \wedge \dx{j}$ and $P\beta = \tfrac{1}{2} (P\beta)_{ab} \dx{a} \dx{b}$, then one can show~\cite[Section 2.2]{K-flows} that
\begin{equation} \label{eq:Pcoords}
(P\beta)_{ab} = \tfrac{1}{2} \ps_{abcd} g^{ci} g^{dj} \beta_{ij}.
\end{equation}
That is, up to the factor of $\tfrac{1}{2}$, the map $P$ is given by contracting the $2$-form with the $4$-form $\ps$ on two indices. It is easy to check that $P$ is self-adjoint and thus orthogonally diagonalizable with real eigenvalues. Using the fundamental identity~\eqref{eq:psps2} for the contraction of $\ps$ with itself on two indices, we find
\begin{align*}
(P^2 \beta)_{ab} & = \tfrac{1}{2} \ps_{abcd} g^{ci} g^{dj} (P \beta)_{ij} = \tfrac{1}{4} \ps_{abcd} g^{ci} g^{dj} \ps_{ijst} g^{sp} g^{tq} \beta_{pq} \\
& = \tfrac{1}{4} ( 4 g_{as} g_{bt} - 4 g_{at} g_{bs} - 2 \ps_{abst} ) g^{sp} g^{tq} \beta_{pq} \\
& = \beta_{ab} - \beta_{ba} - \tfrac{1}{2} \ps_{abst} g^{sp} g^{tq} \beta_{pq} = 2 \beta_{ab} - (P\beta)_{ab}.
\end{align*}
Thus we deduce that $P^2 = 2I - P$, so $(P+2I)(P-I) = 0$. Thus the eigenvalues of $P$ are $-2$ and $+1$, in agreement with~\eqref{eq:Omega2}. To verify that $\lambda = -2$ corresponds to $\Omega^2_7$ as given in~\eqref{eq:Omega2}, we let $\beta_{ij} = (X \hk \ph)_{ij} = X^m \ph_{mij}$. Then using~\eqref{eq:phps2} we have
\begin{equation*}
(P \beta)_{ab} = \tfrac{1}{2} \ps_{abcd} g^{ci} g^{dj} X^m \ph_{mij} = - 2 X^m \ph_{mab} = - 2 \beta_{ab},
\end{equation*}
as claimed. Also, the condition that $\Omega^2_{14} = (\Omega^2_{7})^{\perp}$ gives that $\beta \in \Omega^2_{14}$ must satisfy $X^m \ph_{mij} \beta_{ab} g^{ia} g^{jb} = 0$ for all $X^m$. This is equivalent to $\ph_{mij} \beta_{ab} g^{ia} g^{jb} = 0$. Thus, we can describe the decomposition~\eqref{eq:Omega2} of $\Omega^2$ in local coordinates as
\begin{equation} \label{eq:Omega2-coords}
\begin{aligned}
\beta_{ij} \in \Omega^2_7 & \quad \Longleftrightarrow \quad \beta_{ij} = X^m \ph_{mij} \qquad \Longleftrightarrow \quad \tfrac{1}{2} \ps_{abcd} g^{ci} g^{dj} \beta_{ij} = -2 \beta_{ab}, \\
\beta_{ij} \in \Omega^2_{14} & \quad \Longleftrightarrow \quad \beta_{ij} g^{ia} g^{jb} \ph_{abc} = 0  \quad \Longleftrightarrow \quad \tfrac{1}{2} \ps_{abcd} g^{ci} g^{dj} \beta_{ij} = \beta_{ab}.
\end{aligned}
\end{equation}
Moreover, it is easy to check using~\eqref{eq:phph2} that for $\beta \in \Omega^2_7$ we have
\begin{equation} \label{eq:Omega27}
\beta_{ij} = X^m \ph_{mij} \quad \Longleftrightarrow \quad X^m = \tfrac{1}{6} \beta_{ab} g^{ai} g^{bj} \ph_{ijk} g^{km}.
\end{equation}

\begin{rmk} \label{rmk:Omega2}
The description of the orthogonal splitting $\Omega^2 = \Omega^2_7 \oplus \Omega^2_{14}$ in terms of the $-2, +1$ eigenspaces of the operator $\beta \mapsto \sta (\ph \wedge \beta)$ is analogous to the orthogonal splitting $\Omega^2 = \Omega^2_+ \oplus \Omega^2_-$ into self-dual and anti-self-dual $2$-forms on an oriented Riemannian $4$-manifold with respect to the operator $\beta \mapsto \sta \beta$. This analogy is important in $\G$ gauge theory.
\end{rmk}

\begin{rmk} \label{rmk:signs}
Many authors prefer to use the opposite orientation than we do for the orientation induced by $\ph$. (See~\cite{K-notes} for more details.) This changes the sign of $\sta$. The upshot is that the eigenvalues $(-2,+1)$ in~\eqref{eq:Omega2} and~\eqref{eq:Omega2-coords} are replaced by $(+2,-1)$. Readers should take care to be aware of any particular paper's sign conventions.
\end{rmk}

The local coordinate description of the decomposition~\eqref{eq:Omega3} of $\Omega^3$ can be
understood by considering the infinitesimal action of the $(1,1)$ tensors $\Gamma(T^* M \otimes TM)$ on $\ph$. Let $A = A^i_l \in \Gamma(T^* M \otimes TM)$. At each point $p \in M$, we have $e^{At} \in \GL{T_p M}$, and thus
\begin{equation} \label{eq:diamond-temp}
e^{A t} \cdot \ph = \tfrac{1}{6} \ph_{ijk} \, (e^{At} \dx{i}) \wedge (e^{At} \dx{j}) \wedge (e^{At} \dx{k}).
\end{equation}
Define $A \diamond \ph \in \Omega^3$ by
\begin{equation} \label{eq:diamond}
(A \diamond \ph) = \left. \ddt \right|_{t=0} (e^{At} \cdot \ph).
\end{equation}
From~\eqref{eq:diamond-temp} we compute
\begin{equation*}
(A \diamond \ph) = \tfrac{1}{6} ( A^l_i \ph_{ljk} + A^l_j \ph_{ilk} + A^l_k \ph_{ijl} ) \dx{i} \wedge
\dx{j} \wedge \dx{k},
\end{equation*}
and hence
\begin{equation} \label{eq:diamond-coords}
(A \diamond \ph)_{ijk} = A^l_i \ph_{ljk} + A^l_j \ph_{ilk} + A^l_k \ph_{ijl}.
\end{equation}

Use the metric $g$ to identify $A \in \Gamma(T^* M \otimes TM)$ with a bilinear form $A \in \Gamma(T^* M \otimes T^* M)$ by $A_{ij} = A^l_i g_{lj}$. Recall from Section~\ref{sec:notation} that there is an orthogonal splitting
\begin{equation*}
\Gamma(T^* M \otimes T^* M) \cong \Omega^0 \oplus \Symo \oplus \Omega^2.
\end{equation*}
By the orthogonal decomposition~\eqref{eq:Omega2} on $\Omega^2$ discussed above, we can further decompose this as
\begin{equation*}
\Gamma(T^* M \otimes T^* M) \cong \Omega^0 \oplus \Symo \oplus \Omega^2_7 \oplus \Omega^2_{14}.
\end{equation*}
With respect to this splitting, we can write $A = \tfrac{1}{7} (\tr A) g + A_0 + A_7 + A_{14}$, where $A_0$ is a traceless symmetric tensor.

By~\eqref{eq:diamond-coords}, we have a \emph{linear map} $A \mapsto A \diamond \ph$ from $\Omega^0 \oplus \Symo \oplus \Omega^2_7 \oplus \Omega^2_{14}$ to the space $\Omega^3$.

\begin{prop} \label{prop:Omega3}
The kernel of $A \mapsto A \diamond \ph$ is $\Omega^2_{14}$, and the remaining three summands $\Omega^0$, $\Symo$, $\Omega^2_7$, of $\Gamma(T^* M \otimes T^* M)$ are mapped isomorphically onto $\Omega^3_1$, $\Omega^3_{27}$, $\Omega^3_7$, respectively. Explicitly, if $A = \tfrac{1}{7} (\tr A) g + A_0 + A_7 + A_{14}$, then
\begin{equation*}
A \diamond \ph = \underbrace{\tfrac{3}{7} (\tr A) \ph}_{\Omega^3_1} + \underbrace{A_0 \diamond \ph}_{\Omega^3_{27}} + \underbrace{X \hk \ps}_{\Omega^3_7},
\end{equation*}
where
\begin{equation*}
X^m = -\tfrac{1}{2} A_{ij} g^{ia} g^{jb} \ph_{abc} g^{cm}.
\end{equation*}
\end{prop}
\begin{proof}
This can be established using the various contraction identities of Theorem~\ref{thm:G2-identities}. The explicit details can be found in~\cite[Section 2.2]{K-flows}.
\end{proof}

\begin{rmk} \label{rmk:Omega214}
The fact that $\Omega^2_{14}$ is the kernel of $A \mapsto A \diamond \ph$ is a consequence of the fact that $\G$ is the Lie group that preserves $\ph$. Thus the infinitesimal action, which is the action of the Lie algebra $\mathfrak{g}_2$, annihilates $\ph$. This is consistent with the fact that $\G \subset \SO{7}$, so $\mathfrak{g}_2 \subset \mathfrak{so}(7) \cong \Lambda^2 (\R^7)^*$. Thus, at every point $p \in M$, the space $\Lambda^2_{14} (T_p^* M)$ is isomorphic to $\mathfrak{g}_2$.
\end{rmk}

\subsection{The torsion of a $\G$-structure} \label{sec:torsion}

Let $(M, \ph)$ be a manifold with $\G$-structure. Since $\ph$ determines a Riemannian metric $\ph$, we get a Levi-Civita covariant derivative $\nabla$. Thus it makes sense to consider the tensor $\nabla \ph \in \Gamma(T^* M \otimes \Lambda^3 T^* M)$.

\begin{defn} \label{defn:torsion-free}
The $\G$-structure $\ph$ is called \emph{torsion-free} if $\nabla \ph = 0$. Although this appears to be a linear equation, recall that $\nabla$ is induced from $g$ which itself depends nonlinearly on $\ph$. Thus the equation $\nabla \ph = 0$ is in fact a fully nonlinear first order partial differential equation for $\ph$. We say $(M, \ph)$ is a \emph{torsion-free $\G$ manifold} if $\ph$ is a torsion-free $\G$-structure on $M$. For brevity, we sometimes use the term ``$\G$ manifold'' when we mean ``torsion-free $\G$ manifold''.
\end{defn}

The fundamental observation about the torsion of any $\G$-structure is the following.

\begin{thm} \label{thm:torsion-symmetry}
Let $X$ be a vector field on $M$. Then the $3$-form $\nabla_X \ph$ lies in the subspace
$\Omega^3_7$ of $\Omega^3$. Thus the covariant derivative $\nabla \ph$ is a smooth section of $T^* M \otimes \Lambda^3_7 (T^* M)$.
\end{thm}
\begin{proof}
By Proposition~\ref{prop:Omega3}, any $3$-form $\gamma$ can be written as $\gamma = A \diamond \ph$ for a unique $A = \tfrac{1}{7} (\tr A) g + A_0 + A_7$. We take the inner product of $A \diamond \ph$ with $\nabla_X \ph$. Using~\eqref{eq:diamond-coords}, this is
\begin{align*}
\langle A \diamond \ph, \nabla_X \ph \rangle & = \tfrac{1}{6} (A \diamond \ph)_{ijk} (\nabla_X \ph)_{abc} g^{ia} g^{jb} g^{kc} \\
& = \tfrac{1}{6} (A^l_i \ph_{ljk} + A^l_j \ph_{ilk} + A^l_k \ph_{ijl}) X^m \nabla_m \ph_{abc} g^{ia} g^{jb} g^{kc} \\
& = \tfrac{1}{2} A^l_i \ph_{ljk} X^m \nabla_m \ph_{abc} g^{ia} g^{jb} g^{kc} \\
& = \tfrac{1}{2} A_{ip} X^m \ph_{qjk} \nabla_m \ph_{abc} g^{pq} g^{ia} g^{jb} g^{kc}.
\end{align*}
Taking the covariant derivative of~\eqref{eq:phph2} and using that $g$ is parallel, we get
\begin{equation*}
\nabla_m \ph_{qjk} \ph_{abc} g^{jb} g^{kc} = - \ph_{qjk} \nabla_m \ph_{abc} g^{jb} g^{kc}.
\end{equation*}
This says that $\nabla_m \ph_{qjk} \ph_{abc} g^{jb} g^{kc}$ is skew in $q,a$. Thus the symmetric part of $A_{ip}$ does not contribute to $\langle A \diamond \ph, \nabla_X \ph \rangle$ above. That is, $\nabla_X \ph$ is orthogonal to any element of $\Omega^3_1 \oplus \Omega^3_{27}$, as claimed.
\end{proof}

Theorem~\ref{thm:torsion-symmetry} motivates the following definition.

\begin{defn} \label{defn:torsion}
Because $\nabla_X \ph \in \Omega^3_7$, by~\eqref{eq:Omega3} we can write
\begin{equation*}
\nabla_X \ph = T(X) \hk \ps
\end{equation*}
for some vector field $T(X)$ on $M$. That is, there exists a tensor $T \in \Gamma(T^* M \otimes T^* M)$ such that
\begin{equation} \label{eq:torsion}
\nabla_m \ph_{ijk} = T_{mp} g^{pq} \ps_{qijk}.
\end{equation}
We call $T$ the \emph{full torsion tensor} of $\ph$.
\end{defn}

By contracting~\eqref{eq:torsion} with $\ps_{nijk}$ on $i,j,k$ and using~\eqref{eq:psps3}, we obtains
\begin{equation} \label{eq:torsion-2}
T_{mn} = \tfrac{1}{24} \nabla_m \ph_{ijk} \ps_{nabc} g^{ia} g^{jb} g^{kc}.
\end{equation}
Moreover, taking the covariant derivative of~\eqref{eq:phph1} and using~\eqref{eq:torsion} and~\eqref{eq:phps1}, one can compute that
\begin{equation} \label{eq:nabla-ps}
\nabla_p \ps_{ijkl} = - T_{pi} \ph_{jkl} + T_{pj} \ph_{ikl} - T_{pk} \ph_{ijl} + T_{pl} \ph_{ijk}.
\end{equation}

Observe that equations~\eqref{eq:torsion} and~\eqref{eq:torsion-2} show that $\nabla \ph = 0$ if and only if $T = 0$. (In this case~\eqref{eq:nabla-ps} shows that $\nabla \ps = 0$ as well, which is also clear because $\ps = \sta \ph$ and $\nabla$ commutes with $\sta$.)

Hence $\ph$ is torsion-free if and only if $T = 0$. The tensor $T$ is a more convenient measure of the failure of $\ph$ to be parallel, because we can easily decomposes it into four independent pieces in $\Gamma(T^* M \otimes T^* M) \cong \Omega^0 \oplus \Symo \oplus \Omega^2_7 \oplus \Omega^2_{14}$, as
\begin{equation} \label{eq:torsion-decomp}
T = T_1 + T_{27} + T_7 + T_{14},
\end{equation}
where $T_1 = \tfrac{1}{7} (\tr T) g$ and $T_{27} = T_0$ is the traceless symmetric part of $T$.

\begin{cor} \label{cor:FG}
Let $\ph$ be a $\G$-structure on $M$. Then $\ph$ is torsion-free if and only if both $\dd \ph = 0$ and $\dd \ps = 0$.
\end{cor}
\begin{proof}
Note that $\dd \ps = \dd \sta \ph = - \sta \ds \ph$, so $\dd \ps = 0$ if and only if $\ds \ph = 0$. Because both the exterior derivative $\dd$ and its adjoint $\ds$ can be written in terms of $\nabla$, any parallel form is always closed and coclosed. It is the converse that is nontrivial here. In fact, $\dd \ph$ and $\ds \ph$ are \emph{linear} in $\nabla \ph$ and hence linear in $T$. Since $\dd \ph \in \Omega^4 = \Omega^4_1 \oplus \Omega^4_7 \oplus \Omega^4_{27}$ and $\ds \ph \in \Omega^2 = \Omega^2_7 \oplus \Omega^2_{14}$, it follows by Schur's Lemma that the independent components of $\dd \ph$ and $\ds \ph$ must correspond to the $1, 7, 14, 27$ components of $T$ as in~\eqref{eq:torsion-decomp}, up to constant factors. Thus if $\dd \ph = 0$ and $\ds \ph = 0$, we must have $T = 0$.
\end{proof}

Corollary~\ref{cor:FG} is a classical theorem of Fern\`andez--Gray~\cite{FG}. The present proof is an extremely abridged version of the argument in~\cite[Section 2.3]{K-flows}.

\begin{rmk} \label{rmk:harmonic}
Recall that a differential form $\gamma$ on $(M, g)$ is \emph{harmonic} if $\Delta_{\dd} \gamma = (\dd \ds + \ds \dd) \gamma = 0$. On a compact manifold, by integration by parts harmonicity is equivalent to $\dd \gamma = 0$ and $\ds \gamma = 0$. Thus Corollary~\ref{cor:FG} says that in the compact case, a $\G$-structure $\ph$ is torsion-free if and only if it is harmonic \emph{with respect to its induced metric}.
\end{rmk}

Since the torsion $T$ of $\ph$ decomposes into four independent components as in~\eqref{eq:torsion-decomp}, each component can be zero or nonzero. This gives $2^4 = 16$ distinct classes of $\G$-structures. Some of the more interesting classes of $\G$-structures are given in the following table.
{\renewcommand{\arraystretch}{1.2}
\begin{center}
\begin{tabular}{|c|c|c|c|c|}
\hline
$T_1$ & $T_{27}$ & $T_7$ & $T_{14}$ & \emph{$\G$-structure} \\ \hline
$0$ & $0$ & $0$ & $0$ & $\nabla \ph = 0$ (torsion-free) \\ \hline
$0$ & $0$ & $0$ & $\ast$ & $\dd \ph = 0$ (closed) \\ \hline
$\ast$ & $\ast$ & $0$ & $0$ & $\ds \ph = 0$ (coclosed) \\ \hline
$\ast$ & $0$ & $0$ & $0$ & $\dd \ph = \lambda \ps$ ($\lambda \neq 0$) \\ \hline
\end{tabular}
\end{center}
}
The last class in the table above is called \emph{nearly parallel}, and one can show that $\lambda$ is constant and that the induced metric is \emph{positive Einstein}, with $R_{ij} = \tfrac{3}{8} \lambda^2 g_{ij}$. (For example, see~\cite[after Remark 4.19]{K-flows}.)

More details on the 16 classes of $\G$-structures can be found in~\cite{CS, FG, K1, K-flows}. In particular it is worth remarking~\cite[Theorem 2.32]{K-flows} that with respect to conformal changes of $\G$-structure, the component $T_7$ plays a very different role than the other three components $T_1, T_{27}, T_{14}$.

\textbf{Aside.} There is an equivalent approach to studying $\G$-structures using \emph{spin geometry}. Let $(M^7, g)$ be a Riemannian $7$-manifold equipped with a spin structure and associated \emph{spinor bundle} $\spi (M)$. This is a real rank $8$ vector bundle over $M$. Since $8>7$, by algebraic topology, this bundle always admits nowhere vanishing sections. Such a section $s$ determines a $3$-form $\ph$ on $M$ by $\ph(a,b,c) = \langle a \cdot b \cdot c \cdot s, s \rangle$, where $\cdot$ denotes the Clifford multiplication of tangent vectors to $M$ on spinors. Using the fact that $s$ is nowhere zero, one can show that the $3$-form $\ph$ is always a $\G$-structure. Moreover, $\ph$ is torsion-free if and only if $s$ is a \emph{parallel spinor}, with respect to the spin connection on $\spi (M)$ induced from the Levi-Civita connection of $g$. (The existence of a parallel spinor for torsion-free $\G$ manifolds is precisely why they are of interest in theoretical physics, as this is related to \emph{supersymmetry}.) Similarly, $\ph$ is nearly parallel in the sense defined above if and only if $s$ is a \emph{Killing spinor}. The reader is directed to Harvey~\cite{Harvey}, Lawson--Michelsohn~\cite[Chapter IV. 10]{LM}, and the more recent paper by Agricola--Chiossi--Friedrich--H\"oll~\cite{ACFH} for more on this point of view. This approach is also very important in the construction of \emph{invariants} of $\G$-structures, as discussed by Crowley--Goette--Nordstr\"om~\cite{CGN} in the present volume.

\subsection{Relation between curvature and torsion for a $\G$-structure} \label{sec:curvature}

Let $(M, \ph)$ be a manifold with $\G$-structure. Since $\ph$ determines a Riemannian metric $g_{\ph}$, we have a Riemann curvature tensor $R$. There is an important relation between the tensors $R$ and $\nabla T$, called the ``\emph{$\G$~Bianchi identity}'' that originally appeared in~\cite[Theorem 4.2]{K-flows}.

\begin{thm} \label{thm:G2-Bianchi}
The \emph{$\G$-Bianchi identity} is the following:
\begin{equation} \label{eq:G2-Bianchi}
\nabla_i T_{jk} - \nabla_j T_{ik} = (T_{ip} T_{jq} + \tfrac{1}{2} R_{ijpq}) g^{pa} g^{qb} \ph_{abk}.
\end{equation}
\end{thm}
\begin{proof}
Equation~\eqref{eq:G2-Bianchi} can be derived by combining the covariant derivative of~\eqref{eq:torsion} with~\eqref{eq:nabla-ps} to get an expression for $\nabla_m \nabla_p \ph_{ijk}$ in terms of $\ph$, $\ps$, and $T$, and $\nabla T$. Then applying the Ricci identity to the difference
\begin{equation*}
\nabla_m \nabla_p \ph_{ijk} - \nabla_p \nabla_m \ph_{ijk}
\end{equation*}
introduces Riemann curvature terms. Simplifying further using the identities of Theorem~\ref{thm:G2-identities} eventually results in~\eqref{eq:G2-Bianchi}.
\end{proof}

An important consequence of Theorem~\ref{thm:G2-Bianchi} is the following.

\begin{cor} \label{cor:Ricci}
The Ricci curvature $R_{jk}$ of the metric $g$ induced by a $\G$-structure $\ph$ can be expressed in terms of the torsion $T$ and its covariant derivative $\nabla T$ as follows:
\begin{equation} \label{eq:Ricci}
\begin{aligned}
R_{jk} & = (\nab{i} T_{jp} - \nab{j} T_{ip}) \ph_{lqk} g^{pq} g^{il} - T_{jp} g^{pq} T_{qk} + (\tr T) T_{jk} \\
& \qquad {} - T_{jl} T_{ab} g^{ap} g^{bq} \ps_{pqmk} g^{lm}.
\end{aligned}
\end{equation}
\end{cor}
\begin{proof}
Equation~\eqref{eq:Ricci} can be obtained from~\eqref{eq:G2-Bianchi} by combining the first Bianchi identity of Riemannian geometry together with the identities of Theorem~\ref{thm:G2-identities}. The details can be found in~\cite[Section 4.2]{K-flows}.
\end{proof}

\begin{rmk} \label{rmk:Ricci-flat}
Equation~\eqref{eq:Ricci} shows that the metric of a torsion-free $\G$-structure is always \emph{Ricci-flat}. (See also item (vi) of Remark~\ref{rmk:Berger} below.)
\end{rmk}

On a general Riemannian manifold $(M^n, g)$, the Riemann curvature tensor $R$ decomposes into the scalar curvature, the traceless Ricci curvature, and the conformally invariant \emph{Weyl curvature}. When $g$ is induced from a $\G$-structure $\ph$, the Weyl tensor $W$decomposes further intro three independent components $W_{27}$, $W_{64}$, and $W_{77}$ as irreducible $\G$-representations. A detailed discussion of the curvature decomposition of $\G$-structures can be found in Cleyton--Ivanov~\cite{CI} and in the forthcoming~\cite{DGK}.

\section{Exceptional Riemannian holonomy} \label{sec:exceptional-holonomy}

In this section we briefly review the notion of the \emph{holonomy} of a Riemannian manifold $(M, g)$, and place the geometry of torsion-free $\G$-structures in this context, as one of the geometries with \emph{exceptional Riemannian holonomy}.

\subsection{Parallel transport and Riemannian holonomy} \label{sec:review-holonomy}

Let $(M^n, g)$ be a Riemannian manifold, and let $\nabla$ be the Levi-Civita connection of the metric $g$. We review without proof the well-known basic properties of Riemannian holonomy. See, for example,~\cite[Chapters 2 \& 3]{Joyce} for a more detailed discussion.

\begin{defn} \label{defn: holonomy}
Fix $p \in M$. Let $\gamma$ be loop based at $p$. This means that $\gamma : [0,1] \to M$ is a continuous path, and piecewise smooth, such that $\gamma(0) = \gamma(1) = p$. Then, with respect to $\nabla$, the \emph{parallel transport} $\Pi_{\gamma} : T_{\gamma(0)} M \to T_{\gamma(1)} M$ around the loop $\gamma$ is a linear isomorphism of $T_p M$ with itself, which depends on $\gamma$. We define the \emph{holonomy} of the metric $g$ at the point $p$, denoted $\Hol_p (g)$, to be the set of all such isomorphisms. That is,
\begin{equation*}
\Hol_p (g) = \{ \Pi_{\gamma} : T_p M \cong T_p M : \gamma \text{ is a loop based at } p \}.
\end{equation*}
It follows from the existence and uniqueness of parallel transport (which itself is a consequence of existence and uniqueness for systems of first order linear ordinary differential equations) that $\Pi_{\gamma \cdot \beta} = \Pi_{\gamma} \circ \Pi_{\beta}$, where $\gamma \cdot \beta$ is the concatenation of paths, $\beta$ followed by $\gamma$. Consequently, it is easy to see that $\Hol_p(g)$ is closed under multiplication and inversion. That is, $\Hol_p (g)$ is a \emph{subgroup} of $\GL{T_p M}$.

If we instead consider the restricted class of \emph{contractible} loops at $p$, which is closed under concatenation of paths, we obtain the \emph{restricted holonomy} of $g$ at $p$, denoted $\Hol^0_p (g)$. The group $\Hol^0_p (g)$ is a normal subgroup of $\Hol_p (g)$, and is the connected component of the identity. If $M$ is simply-connected, then $\Hol^0_p (g) = \Hol_{p} (g)$ for all $p \in M$.

Because $\nabla$ is the Levi-Civita connection, we have $\nabla g = 0$. Thus parallel transport with respect to $\nabla$ preserves the inner product, and we conclude that in fact $\Hol_p (g)$ is a subgroup of $\Or{T_p M, g_p}$, the group of isometries of the inner product space $(T_p M, g_p)$. Similarly $\Hol^0_p (g)$ is a subgroup of $\SO{T_p M, g_p}$, the group of orientation-preserving isometries of $(T_p M, g_p)$.
\end{defn}

The following proposition is straightforward to prove using the definitions. 

\begin{prop} \label{prop:holonomy-properties}
The holonomy group $\Hol_p (g)$ satisfies the following properties.
\begin{itemize}
\item Let $p, q \in M$ lie in the same connected component of $M$. Then $\Hol_p (g) \cong \Hol_q (g)$. In fact, if $\gamma$ is a piecewise smooth continuous path from $p$ to $q$, and $P = \Pi_{\gamma} : T_p M \cong T_q M$ is the parallel transport isomorphism from $T_p M$ to $T_q M$, then $\Hol_q (g) = P \cdot \Hol_p (g) \cdot P^{-1}$.
\item Fix $p \in M$, and fix an isomorphism $T_p M \cong \R^n$. Then $\GL{T_p M} \cong \GL{n}$ and $\Or{T_p M, g_p} \cong \Or{n}$. With respect to this identification, $\Hol_p (g)$ corresponds to a subgroup $H \subseteq \Or{n}$. If we choose any other isomorphism $T_p M \cong \R^n$, then the resulting subgroup $\tilde H$ of $\Or{n}$ is in the same conjugacy class as $H$. 
\item Suppose $M$ is connected. Then $\Hol_p (g) \cong \Hol_q (g)$ for all $p, q \in M$. Moreover, there exists a subgroup $H$ of $\Or{n}$ such that $\Hol_p (g) \cong H$ for all $p \in M$, and this subgroup $H$ is unique up to conjugation.
\end{itemize}

Analogous statements hold for the restricted holonomy group $\Hol^0_p (g)$, determining (when $M$ is connected) a subgroup $H^0$ of $\SO{n}$, unique up to conjugation.
\end{prop}

Consequently, if $M$ is connected, we abuse notation and call $H$ \emph{the holonomy group} and $H^0$ \emph{the restricted holonomy group} of $(M, g)$. Observe that $H$ and $H^0$ are not just abstract groups, but that they come naturally equipped with isomorphism classes of \emph{representations on $T_p M$} for all $p \in M$.

Recall that a tensor $S$ on $M$ is called \emph{parallel} if $\nabla S =  0$. There is a fundamental relationship between the holonomy group of $g$ and the \emph{parallel tensors} on $M$, given by the following.

\begin{prop} \label{prop:holonomy-parallel-tensors}
Let $(M, g)$ be a Riemannian manifold. Fix $p \in M$. Let $H \subseteq \GL{T_p M}$ be the subgroup that leaves invariant $S|_p$ for all parallel tensors $S$ on $M$. 
\begin{itemize}
\item We always have $\Hol_p (g) \subseteq H$. Moreover, these two subgroups are \emph{usually equal}. For example, this is the case if $\Hol_p (g)$ is a closed subgroup of $\GL{T_p M}$.
\item If the group $H$ fixes an element $S_0$ in some tensor space of $T_p M$, then there exists a \emph{parallel tensor} $S$ on $M$ such that $S|_p = S_0$.
\end{itemize}
\end{prop}

The way to think about Proposition~\ref{prop:holonomy-parallel-tensors} is as follows. The Riemannian holonomy $H$ of a Riemannian manifold $(M, g)$ is strictly smaller than $\Or{n}$ if and only if there exist nontrivial parallel tensors on $M$ other than the metric $g$.

\begin{rmk} \label{rmk:orientable}
If $M$ is simply-connected, then $H = H^0$ and consequently $H \subseteq \SO{n}$. This means there exists a (necessarily parallel) Riemannian volume form $\mu \in \Omega^n (M)$ on $M$. This is consistent with the well-known fact from topology that any simply-connected manifold is orientable.
\end{rmk}

\subsection{The Berger classification of Riemannian holonomy} \label{sec:Berger}

In 1955, Marcel Berger classified the possible Lie subgroups of $\Or{n}$ that could occur as Riemannian holonomy groups of a metric $g$, subject to the following technical hypotheses.
\begin{itemize}
\item We restrict to \emph{simply-connected} manifolds. In general if $(M, g)$ is not simply-connected then the holonomy $H$ of $(M, g)$ is a finite cover of the reduced holonomy $H^0$. That is, the quotient $H/H^0$ is a discrete group.
\item We must exclude the case when $(M, g)$ is \emph{locally reducible}. A locally reducible Riemannian manifold is \emph{locally} a Riemannian product $(M_1, g_1) \times (M_2, g_2)$. In this case the Riemanian holonomy of $(M, g)$ is a product of the holonomies of $(M_1, g_1)$ and $(M_2, g_2)$.
\item We must exclude the case when $(M, g)$ is \emph{locally symmetric}. A locally symmetric Riemannian manifold is \emph{locally} isometric to a symmetric space $(G/H, g)$ where $G$ is a group of isometries acting transitively on $G/H$ with isotropy group $H$ at any point. In this case the Riemannian holonomy of $(M, g)$ is $H$.
\end{itemize}

\begin{thm}[Berger classification] \label{thm:Berger}
Let $(M, g)$ be a \emph{simply-connected} smooth Riemannian manifold of dimension $n$ that is \emph{not locally reducible} and \emph{not locally symmetric}. Then the Riemannian holonomy $H \subseteq \SO{n}$ can only be one of the following seven possibilities:
{\renewcommand{\arraystretch}{1.2}
\begin{center}
\begin{tabular}{|c|c|c|c|c|}
\hline
$n = \dim M$ & $H$ & \emph{Parallel tensors} & \emph{Name} & \emph{Curvature} \\ \hline
$n$ & $\SO{n}$ & $g, \mu$ & orientable & \\ \hline
$2m$ ($m\geq 2$) & $\U{m}$ & $g, \omega$ & K\"ahler & \\ \hline
$2m$ ($m\geq 2$) & $\SU{m}$ & $g, \omega, \Omega$ & Calabi-Yau & Ricci-flat \\ \hline
$4m$ ($m\geq 2$) & $\Sp{m}$ & $g, \omega_1, \omega_2, \omega_3 , J_1, J_2, J_3$ & hyper-K\"ahler & Ricci-flat \\ \hline
$4m$ ($m\geq 2$) & $(\Sp{m} \times \Sp{1})/ \Z_2$ & $g, \Upsilon$ & quaternionic-K\"ahler & Einstein \\ \hline
$7$ & $\G$ & $g, \ph, \ps$ & $\G$ & Ricci-flat \\ \hline
$8$ & $\Spin{7}$ & $g, \Phi$ & $\Spin{7}$ & Ricci-flat \\ \hline
\end{tabular}
\end{center}
}
\end{thm}
\begin{proof}[Sketch of proof.]
Berger arrived at this classification by studying the \emph{holonomy algebra} $\mathfrak{h}$ of the holonomy group $H$. There is an intimate relation between $\mathfrak{h}$ and the Riemann curvature operator $R \in \Sym^2 (\mathfrak{so}(n))$ of $g$. First, because the Riemann curvature operator can be viewed as ``infinitesimal holonomy'', it must be that $R \in \Sym^2 (\mathfrak{h})$. Since it also satisfies the first Bianchi identity, this says that $\mathfrak{h}$ cannot be too big. Second, by the Ambrose--Singer holonomy theorem, the span of the image of $R$ at any point in $M$ must generate $\mathfrak{h}$ as a vector space, so $\mathfrak{h}$ cannot be too small. Finally, for certain possible $\mathfrak{h}$, the fact that $R$ must also satisfy the second Bianchi identity forces $\nabla R = 0$, in which case $(M, g)$ is locally symmetric. Much more detailed discussion of this argument can be found in Joyce~\cite[Section 3.4]{Joyce}.
\end{proof}

\begin{rmk} \label{rmk:Berger}
We make some remarks concerning the above table.
\begin{enumerate}[(i)]
\item The four restrictions $m \geq 2$ in the first column are mostly to eliminate redundancy, as we have the isomorphisms $\U{1} \cong \SO{2}$, $\Sp{1} \cong \SU{2}$, and $(\Sp{1} \times \Sp{1})/\Z_2 \cong \SO{4}$. The case $\SU{1}$ does not occur because $\SU{1} \cong \{ 1 \}$ is trivial and such a space is flat and thus symmetric.
\item Because $\Sp{k} \subseteq \SU{2k} \subseteq \U{2k}$, all hyper-K\"ahler manifolds are Calabi-Yau, and all Calabi-Yau manifolds are K\"ahler.
\item Note that quaternionic-K\"ahler manifolds are in fact \emph{not} K\"ahler. This ill-advised nomenclature has unfortunately stuck and is here to stay.
\item Usually, the term \emph{special holonomy} refers to any of the holonomy groups above other than the first two, perhaps because K\"ahler manifolds exist in sufficient abundance to not be that special.
\item The last two groups above, namely $\G$ and $\Spin{7}$, are called the \emph{exceptional holonomy groups}. These Lie groups are both intimately related to the octonions $\Oc$. The connection between $\G$ and $\Oc$ is explained in Section~\ref{sec:G2-package} above. The connection between $\Spin{7}$ and $\Oc$ can be found, for example, in Harvey~\cite[Lemma 14.61]{Harvey} or Harvey--Lawson~\cite[Section IV.1.C.]{HL}.
\item The fact that metrics with special holonomy are all Einstein (including Ricci-flat) follows from consideration of the constraints on the Riemann curvature due to its relation with the holonomy algebra $\mathfrak{h}$, as explained in the sketch proof above. (See also Remark~\ref{rmk:Ricci-flat} above for the $\G$ case.) \qedhere
\end{enumerate}
\end{rmk}
 
It is interesting to note that Berger did not actually prove that all these groups \emph{can actually occur} as Riemannian holonomy groups. He only excluded all other possibilities. It was widely suspected that the exceptional holonomies could not actually occur, only they could not be excluded using Berger's method. We now know, of course, that \emph{all} of the possibilities in the above table do occur, both in compact and in complete noncompact examples. See Section~\ref{sec:history} for a brief survey of this history in the case of $\G$.

\section{Torsion-free $\G$ manifolds} \label{sec:TF}

In this section we discuss torsion-free $\G$ manifolds, including a brief history of the search for irreducible examples, the known topological obstructions to existence in the compact case, and a comparison with K\"ahler and Calabi-Yau manifolds.

\subsection{Irreducible and reducible torsion-free $\G$ manifolds} \label{sec:reducible}

Let $(M, \ph)$ be a torsion-free $\G$ manifold. That is, $\ph$ is a torsion-free $\G$-structure as in Definition~\ref{defn:torsion-free}, and thus by Proposition~\ref{prop:holonomy-parallel-tensors} the holonomy $\Hol(g_{\ph})$ of the induced Riemannian metric $g_{\ph}$ lies in $\G$.

\begin{defn} \label{defn:irreducible}
We say $(M,\ph)$ is an \emph{irreducible} torsion-free $\G$ manifold if $\Hol(g_{\ph}) = \G$.
\end{defn}

A torsion-free $\G$ manifold could have \emph{reduced holonomy}. That is, we could have $\Hol(g_{\ph}) \subsetneq \G$. In fact there are some simple constructions that yield such reducible examples:

\begin{itemize}
\item If $g_{\ph}$ is flat, then $\Hol(g_{\ph}) = \{ 1 \}$. In this case $M$ is locally isomorphic to Euclidean $\R^7$ with the standard $\G$-structure $\ph_{\oo}$.
\item Let $L^4$ be a manifold with holonomy $\SU{2} \cong \Sp{1}$. This is a hyper-K\"ahler $4$-manifold with hyper-K\"ahler triple $\omega_1$, $\omega_2$, $\omega_3$. Let $X^3$ be a flat Riemannian $3$-manifold with global orthonormal parallel coframe $e^1, e^2, e^3$. Let $M^7 = X^3 \times L^4$, and define a smooth $3$-form $\ph$ on $M$ by
\begin{equation*}
\ph =  e^1 \wedge \omega_1 + e^2 \wedge \omega_2 + e^3 \wedge \omega_3 - e^1 \wedge e^2 \wedge e^3.
\end{equation*}
Then $\ph$ is a torsion-free $\G$-structure with $\Hol(g_{\ph}) = \SU{2} \subsetneq \G$. In this case we have
\begin{equation*}
\ps = e^2 \wedge e^3 \wedge \omega_1 + e^3 \wedge e^1 \wedge \omega_2 + e^1 \wedge e^2 \wedge \omega_3 - \vol_{L}
\end{equation*}
where $\vol_{L} = \tfrac{1}{2} \omega_1^2 = \tfrac{1}{2} \omega_2^2 = \tfrac{1}{2} \omega_3^2$ is the volume form of $L$.
\item Let $L^6$ be a manifold with holonomy $\SU{3}$. This is a Calabi-Yau complex $3$-fold with K\"ahler form $\omega$ and holomorphic volume form $\Omega$. Let $X^1$ be a Riemannian $1$-manifold with global unit parallel $1$-form $e^1$. Let $M^7 = X^1 \times L^6$, and define a smooth $3$-form $\ph$ on $M$ by
\begin{equation*}
\ph = e^1 \wedge \omega - \real \Omega.
\end{equation*}
Then $\ph$ is a torsion-free $\G$-structure with $\Hol(g_{\ph}) = \SU{3} \subsetneq \G$. In this case we have
\begin{equation*}
\ps =  \tfrac{1}{2} \omega^2 + e^1 \wedge \imag \Omega.
\end{equation*}
\end{itemize}

\begin{rmk} \label{rmk:check-irred}
If $(M, \ph)$ is a torsion-free $\G$ manifold, then some criteria are known to determine if $(M, \ph)$ is irreducible. Here are two examples:
\begin{enumerate}[(i)]
\item If $M$ is \emph{compact} with $\Hol(g_{\ph}) \subseteq \G$, then $\Hol(g_{\ph}) = \G$ if and only if the fundamental group $\pi_1 (M)$ is \emph{finite}. (See Joyce~\cite[Proposition 10.2.2]{Joyce}.)
\item If $M$ is connected and \emph{simply-connected}, with $\Hol(g_{\ph}) \subseteq \G$, then $\Hol(g_{\ph}) = \G$ if and only if there are \emph{no nonzero parallel $1$-forms}. (See Bryant--Salamon~\cite[Theorem 2]{BS}.) \qedhere
\end{enumerate}
\end{rmk}

\subsection{A brief history of irreducible torsion-free $\G$ manifolds} \label{sec:history}

The search for examples of \emph{irreducible} torsion-free $\G$ manifolds (that is, Riemannian metrics with holonomy exactly $\G$) has a long history. As explained in Section~\ref{sec:Berger}, it was originally believed such metrics could not exist. In this section we give a very brief and far from exhaustive survey of some of this history.

The first local (that is, incomplete) examples were found by Bryant~\cite{Bryant-A} in 1987, using methods of exterior differential systems and Cartan-K\"ahler theory.

Then in 1989, Bryant--Salamon~\cite{BS} found the first \emph{complete noncompact} examples of $\G$ holonomy metrics. These were metrics on the \emph{total spaces of vector bundles}. Explicitly, these metrics were found on the bundles $\Lambda^2_- (S^4)$ and $\Lambda^2_- (\C \PR^2)$, which are rank $3$ bundles over $4$-dimensional bases, and on the bundle $\spi(S^3)$, the spinor bundle of $S^3$, which is a rank $4$ bundle over a $3$-dimensional base. These Riemannian manifold are all \emph{asymptotically conical}. That is, the metrics approach Riemannian cone metrics at some particular rate at infinity. These torsion-free $\G$-structures are \emph{cohomogeneity one}. That is, there is a Lie group of symmetries acting on $(M, \ph)$ with generic orbits of codimension one. Such symmetry reduces the partial differential equation $\nabla \ph = 0$ to a (fully nonlinear) system of \emph{ordinary differential equations}, which can be explicitly solved. The fact that the metrics have holonomy exactly $\G$ was verified by using the criterion in item (ii) of Remark~\ref{rmk:check-irred}.

\begin{rmk} \label{rmk:noncompact}
Since then, several explicit examples and a great many nonexplicit examples of complete noncompact holonomy $\G$ metrics have been discovered, with various prescribed asymptotic geometry at infinity, such as asymptotically conical (AC), asymptotically locally conical (ALC), and others. In fact, very recent work of Foscolo--Haskins--Nordstr\"om~\cite{FHN1,FHN2} has produced a spectacular new plethora of such examples.
\end{rmk}

The first construction of \emph{compact} irreducible torsion-free $\G$ manifolds was given by Joyce~\cite{J12} in 1994, and pushed further in the monograph~\cite{Joyce}. The idea is the following. Start with the flat $7$-torus $T^7$, and take the quotient by a discrete group of isometries preserving the $\G$-structure $\ph_{\oo}$. The quotient is a singular \emph{orbifold} with torsion-free $\G$-structure. Joyce then resolved the singularities by gluing in (quasi)-asymptotically locally Euclidean spaces with $\SU{2}$ or $\SU{3}$ holonomy, to produce a smooth compact $7$-manifold $M$ with \emph{closed} $\G$-structure and ``small'' torsion. He then used analysis (see Section~\ref{sec:existence} below) to prove that $M$ admits a torsion-free $\G$-structure. Finally, he showed the metrics had holonomy exactly $\G$ by using the criterion (i) of Remark~\ref{rmk:check-irred}. This first construction is explained in more detail by Kovalev~\cite{Minischool-Kovalev} in the present volume.

The second construction of \emph{compact} irreducible torsion-free $\G$ manifolds was introduced by Kovalev~\cite{Ko} in 2001 and pushed significantly further by Corti--Haskins--Nordstr\"om--Pacini~\cite{CHNP} in 2015. It is called the ``twisted connect sum construction''. The ideas is the following. Start with two noncompact \emph{asymptotically cylindrical} Calabi-Yau complex $3$-folds $L_1$ and $L_2$, which are both asymptotic to $X^4 \times T^2$ where $X^4$ is a K3 complex surface. Take $L_1 \times S^1$ and $L_2 \times S^1$ and glue them together with a ``twist'' by identifying different factors of $S^1$ in order to obtain a smooth compact $7$-manifold. The goal is then to construct a closed $\G$-structure on $M$ with ``small'' torsion that can be perturbed using analysis to a torsion-free $\G$-structure (see Section~\ref{sec:existence} below). Being able to do this is a very delicate problem in algebraic geometry involving ``matching data''. This second construction is also explained in more detail by Kovalev~\cite{Minischool-Kovalev} in the present volume.

More recently, a third construction of \emph{compact} irreducible torsion-free $\G$ manifolds appeared in Joyce--Karigiannis~\cite{JK}, involving glueing $3$-dimensional families of Eguchi-Hanson spaces. This construction differs from the previous two because some of the noncompact ``pieces'' that are being glued together this time \emph{do not} come equipped with torsion-free $\G$-structures. This is dealt with by solving a linear elliptic PDE on the noncompact Eguchi-Hanson space using weighted Sobolev spaces.

All three of the currently known constructions of compact irreducible torsion-free $\G$ manifolds are similar in that they all use \emph{glueing techniques} to construct a closed $\G$-structure $\ph$ with ``small'' torsion, and then invoke a general existence theorem of Joyce to prove that it can be \emph{perturbed} to a nearby torsion-free $\G$-structure $\widetilde \ph$. This existence theorem is the subject of Section~\ref{sec:existence} below.

Thus, we know that Riemannian metrics with holonomy exactly $\G$ \emph{do exist} on compact manifolds, but they are not explicit. This is analogous to the case of Riemannian metrics with holonomy exactly $\SU{m}$ (also called Calabi-Yau metrics) on compact manifolds. By Yau's proof of the Calabi conjecture, we know that many such metrics exist, but we cannot describe them explicitly. In fact, special holonomy metrics on compact manifolds should in some sense be thought of as ``transcendental'' objects.

So far we have only found $\G$-holonomy metrics that are ``close to the edge of the moduli space''. That is, these metrics are close to either developing singularities or tearing apart into two disjoint noncompact pieces. That is, the three known constructions of compact irreducible torsion-free $\G$ manifolds are very likely producing only a very small part of the ``landscape'' of holonomy $\G$ metrics.

\subsection{Cohomological obstructions to existence in the compact case} \label{sec:top}

There are several known \emph{cohomological obstructions} to the existence of torsion-free $\G$-structures on a compact manifold. We describe some of these in this section. Let $(M, \ph)$ be a \emph{compact} manifold with a \emph{torsion-free} $\G$-structure $\ph$. Let $g_{\ph}$ be the Riemannian metric induced by $\ph$. Thus $\Hol(g_{\ph}) \subseteq \G$. Since $(M, g_{\ph})$ is a compact oriented Riemannian manifold, the Hodge Theorem applies. That is, any deRham cohomology class has a unique harmonic representative.

Since $\ph$ is torsion-free, by Corollary~\ref{cor:FG}, the form $\ph$ is closed and coclosed and thus harmonic. Because $\ph \neq 0$, we deduce from the Hodge Theorem that $[\ph]$ is a nontrivial class in $H^3 (M, \R)$. Hence we find our first cohomological obstruction:
\begin{equation*}
b^3 \geq 1 \qquad \text{if $M$ admits a torsion-free $\G$-structure}.
\end{equation*}
where $b^k = \dim H^k (M, \R)$ is the $k^{\text{th}}$ Betti number of $M$. The same argument applies to $\ps$, so $b^4 \geq 1$, but $b^4 = b^3$ by Poincar\'e duality, so this is not new information.

Suppose $\Hol(g_{\ph}) = \G$. Then by item (i) of Remark~\ref{rmk:check-irred} we must have $\pi_1 (M)$ is finite. It follows from algebraic topology that $H^1 (M, \R) = \{ 0 \}$. Hence we find our second cohomological obstruction:
\begin{equation} \label{eq:b1}
b^1 = 0 \qquad \text{if $M$ admits an \emph{irreducible} torsion-free $\G$-structure}.
\end{equation}

Before we can discuss the two other cohomological obstructions, we need to explain the interaction of the representation-theoretic decompositions of Section~\ref{sec:forms} with the Hodge Theorem.

Because $\ph$ is torsion-free, one can show that the Hodge Laplacian $\Delta_{\dd}$ \emph{commutes} with the orthogonal projection operators onto the irreducible summands of the decomposition of $\Omega^{\bu}$ described in Section~\ref{sec:forms}. (See Joyce~\cite[Theorem 3.5.3]{Joyce} for details.) Combining this fact with the Hodge Theorem, we conclude that the decompositions of Section~\ref{sec:forms} \emph{descend to deRham cohomology}. That is, if we define
\begin{equation*}
\mathcal H^k_l = \{ \gamma \in \Omega^k_l \mid \Delta_{\dd} \gamma = 0 \}
\end{equation*}
to be the space of \emph{harmonic $\Omega^k_l$-forms}, and $\mathcal H^k$ to be the space of harmonic $k$-forms, then we have
\begin{equation} \label{eq:harmonic-forms}
\begin{aligned}
\mathcal H^2 & = \mathcal H^2_7 \oplus \mathcal H^2_{14}, \\
\mathcal H^3 & = \mathcal H^3_1 \oplus \mathcal H^3_7 \oplus \mathcal H^3_{27}.
\end{aligned}
\end{equation}
Moreover, it follows from the explicit descriptions of $\Omega^k_l$ in Section~\ref{sec:forms} and the fact that $\Delta_{\dd}$ commutes with the projections and with the Hodge star $\sta$ that
\begin{equation*}
\Delta_{\dd} (f \ph) = (\Delta_{\dd} f) \ph, \qquad \Delta_{\dd}( \alpha \wedge \ph) = (\Delta_{\dd} \alpha) \wedge \ph,
\end{equation*}
for all $f \in \Omega^0$ and all $\alpha \in \Omega^1$. These identities imply that
\begin{align*}
\mathcal H^0_1 & \cong \mathcal H^3_1 \cong \mathcal H^4_1 \cong \mathcal H^7_1, \qquad
\mathcal H^1_7 \cong \mathcal H^2_7 \cong \mathcal H^3_7 \cong \mathcal H^4_7 \cong \mathcal H^5_7 \cong \mathcal H^6_7, \\
\mathcal H^2_{14} & \cong \mathcal H^5_{14}, \hskip 0.9in
\mathcal H^3_{27} \cong \mathcal H^4_{27}.
\end{align*}
Define $b^k_l = \dim \mathcal H^k_l$ to be the ``refined Betti numbers'' of $(M, \ph)$. Then we have shown that
\begin{equation*}
b^k_l = b^{k'}_{l'} \quad \text{if $l = l'$}.
\end{equation*}
In particular $b^k_7 = b^1_7 = b^1$ for $k = 2, \ldots, 6$, and $b^k_1 = b^0_1 = b^0$ for $k=3,4,7$. We deduce that
\begin{align*}
b^2 & = b^2_7 + b^2_{14} = b^1 + b^2_{14}, \\
b^3 & = b^3_1 + b^3_7 + b^3_{27} = b^0 + b^1 + b^3_{27}.
\end{align*}
(Note that if $M$ is connected then $b^0 = 1$, and if in addition $\ph$ is irreducible then by~\eqref{eq:b1} we get $b^2 = b^2_{14}$ and $b^3 = 1 + b^3_{27}$.)

There exists a real quadratic form $Q$ on $H^2(M, \R)$ defined as follows. Let $[\beta] \in H^2 (M, \R)$. Define
\begin{equation} \label{eq:Q}
Q([\beta]) = \int_M \beta \wedge \beta \wedge \ph.
\end{equation}
(In fact, it is easy to see using Stokes's Theorem that $Q$ is well-defined as long as $\dd \ph = 0$, and that $Q$ depends only on the \emph{cohomology class} $[\ph] \in H^3(M, \R)$. We do not need torsion-free to define $Q$.)

But now suppose that $\ph$ is not only torsion-free, but also irreducible. Then by~\eqref{eq:b1} and the discussion above, we have $\mathcal H^2 = \mathcal H^2_{14}$. Given a cohomology class $[\beta] \in H^2(M, \R)$, the Hodge theorem gives us a unique harmonic representative $\beta_H$, which must lie in $\Omega^2_{14}$. By~\eqref{eq:Omega2}, we have $\beta_H \wedge \ph = \ast \beta_H$, and hence
\begin{equation*}
Q([\beta]) = \int_M \beta_H \wedge \beta_H \wedge \ph = \int_M \beta_H \wedge \ast \beta_H = \int_M \| \beta_H \|^2 \vol \geq 0
\end{equation*}
with equality if and only if $\beta_H = 0$, which is equivalent to $[\beta] = 0$. Hence we find our third cohomological obstruction:
\begin{itemize}
\item Let $\ph$ be a \emph{closed} $\G$-structure on a compact manifold $M$ with $\pi_1 (M)$ finite. (So that any torsion-free $\G$-structure on $M$ must necessarily be irreducible.) If there exists a torsion-free $\G$-structure in the cohomology class $[\ph] \in H^3(M, \R)$, then the quadratic form $Q$ defined in~\eqref{eq:Q} must be \emph{positive definite}.
\end{itemize}

\begin{rmk} \label{rmk:signQ}
Recall Remark~\ref{rmk:signs}. If the other orientation is used, then $Q$ must be \emph{negative definite}. So we can unambiguously state the third cohomological obstruction as saying that $Q$ must be \emph{definite}. Moreover, if $\ph$ is merely torsion-free but not irreducible, it is easy to see from~\eqref{eq:Omega2} that (with our convention for the orientation), the \emph{signature} of $Q$ is $(b^2 - b^1, b^1)$.
\end{rmk}

Finally, recall from Chern--Weil theory that a compact $7$-manifold $M$ has a \emph{real first Pontryagin class} $p_1(TM) \in H^4(M, \R)$, defined as the cohomology class represented by the closed $4$-form $\tfrac{1}{8 \pi^2} \tr (R \wedge R)$ where $R \in \Gamma( \End(TM) \otimes \Lambda^2 T^* M)$ is the curvature form of \emph{any} connection on $TM$. If $\ph$ is torsion-free, then $g_{\ph}$ has holonomy contained in $\G$, and hence, because Riemann curvature is ``infinitesimal holonomy'' we have that in fact $R \in \Gamma( \End(TM) \otimes \Lambda^2_{14} T^* M)$. That is, the $2$-form part of $R$ lies in $\Omega^2_{14}$. But then by~\eqref{eq:Omega2} we have
\begin{equation*}
\tr(R \wedge R) \wedge \ph = \tr( R \wedge \ast R) = |R|^2 \vol,
\end{equation*}
and thus
\begin{equation*}
( p_1(TM) \cup [\ph] ) \cdot [M] = \frac{1}{8 \pi^2} \int_M \tr(R \wedge R) \wedge \ph = \frac{1}{8 \pi^2} \int_M | R |^2 \vol,
\end{equation*}
where $[M] \in H_7 (M)$ is the fundamental class of $M$ and $\cdot$ denotes the canonical pairing between $H^7(M, \R)$ and $H_7 (M)$. This is clearly positive unless $R$ is identically zero.  Hence we find our fourth cohomological obstruction:
\begin{equation*}
p_1(TM) \neq 0 \qquad \text{if $M$ admits a \emph{nonflat} torsion-free $\G$-structure}.
\end{equation*}

\subsection{Comparison with K\"ahler and Calabi-Yau manifolds} \label{sec:comparison}

In this section we make some remarks about the similarities and the differences between torsion-free $\G$ manifolds and K\"ahler manifolds in general and Calabi-Yau manifolds in particular. A good reference for K\"ahler and Calabi-Yau geometry is Huybrechts~\cite{Huybrechts}.

Manifolds with $\U{m}$-structure are in some ways analogous to manifolds with $\G$-structure, as detailed in the following table.
{\renewcommand{\arraystretch}{1.2}
\begin{center}
\begin{tabular}{|c|c|c|c|c|}
\hline
& $\U{m}$-structure on $(M^{2m}, g)$ & $\G$-structure on $(M^7, g)$ \\ \hline
Nondegenerate form & $\omega \in \Omega^2$ & $\ph \in \Omega^3$ \\
Vector cross product & $J \in \Gamma(TM \otimes T^* M)$ & $\times \in \Gamma(TM \otimes \Lambda^2 T^* M)$ \\
Fundamental relation & $\omega(u,v) = g(Ju, v)$ & $\ph(u,v,w) = g(u \times v, w)$ \\ \hline
\end{tabular}
\end{center}
}
One very important difference between $\U{m}$-structures and $\G$-structures was already mentioned in Remark~\ref{rmk:no-decouple}, but it is so crucial that it is worth repeating here. For a $\U{m}$-structure, the metric $g$ and the nondegenerate $2$-form $\omega$ are essentially independent, subject only to mild compatibility conditions, and together they determine $J$. In contrast, for a $\G$-structure the nondegenerate $3$-form $\ph$ \emph{determines} the metric $g$ and consequently the cross product $\times$ as well.

Now consider the \emph{torsion-free} cases of such structures. A $\U{m}$-structure is torison-free if $\nabla \omega = 0$. Such manifolds are called \emph{K\"ahler} and have Riemannian holonomy contained in the Lie subgroup $\U{m}$ of $\SO{2m}$. A $\G$-structure is torsion-free if $\nabla \ph = 0$. Such manifolds have Riemannian holonomy contained in the Lie subgroup $\G$ of $\SO{7}$. In the torsion-free cases, both $\omega$ and $\ph$ are \emph{calibrations}. (See~\cite{Minischool-Lotay-calib, Minischool-Leung} in the present volume for more about calibrations.) Both K\"ahler manifolds and torsion-free $\G$ manifolds also admit special \emph{connections} on vector bundles, namely the Hermitian-Yang-Mills connections and the $\G$-instantons, respectively.

Here is where we see another very important difference. As we saw in Remark~\ref{rmk:Ricci-flat}, the metric $g_{\ph}$ of a torsion-free $\G$-structure is always Ricci-flat. But the metric $g$ of a K\"ahler manifold need \emph{not} be Ricci-flat. In fact, the Calabi--Yau Theorem, gives a topological characterization (in the compact case) of exactly which K\"ahler metrics are Ricci-flat. They are precisely those metrics with holonomy contained in the Lie subgroup $\SU{m}$ of $\U{m}$. The precise statement of the Calabi--Yau theorem is as follows.

\begin{thm} \label{thm:Calabi-Yau}
Let $M$ be a \emph{compact} K\"ahler manifold, with K\"ahler form $\omega$. Then there exists a \emph{Ricci-flat} K\"ahler metric $\widetilde{\omega}$ in the same cohomology class as $\omega$ if and only if $c_1(TM) = 0$, where $c_1 (TM)$ is the \emph{first Chern class} of $TM$. Moreover, when it exists the Ricci-flat K\"ahler metric is unique.
\end{thm}

We are very far from having an analogous theorem in $\G$ geometry. In fact, we do not even have any idea of what the correct conjecture might be. The main tool that allowed Yau to reformulate the Calabi conjecture into a statement about existence and uniqueness of solutions to a complex Monge-Amp\`ere equation is the $\del \bar{\del}$-lemma in K\"ahler geometry. There is no close analogue of this result for torsion-free $\G$ manifolds.

Heuristically, the Calabi--Yau Theorem allows us to go from $\U{m}$ holonomy to $\SU{m}$ holonomy, which is a reduction in the dimension of the holonomy group from $m^2$ to $m^2 - 1$, a difference of $1$ dimension, and it corresponds to an (albeit fully nonlinear) \emph{scalar} partial differential equation. In contrast, to obtain a Riemannian metric with holonomy $\G$, we must start with $\SO{7}$ holonomy. Thus we need to reduce the dimension of the holonomy group from $21$ to $14$, so we expect a system of  $7$ \emph{equations}, or equivalently a single partial differential equation for an unknown \emph{1-form} rather than for an unknown function as in the Calabi--Yau Theorem. Precisely how such a heuristic discussion can be made into a precise mathematical conjecture remains a mystery at present.

In fact, a better analogy is the following. Let $M^{2m}$ be a compact manifold that admits $\U{m}$-structures. What are \emph{necessary and sufficient} topological conditions that ensure that $M^{2m}$ admits a K\"ahler structure? We know many necessary conditions. (See Huybrechts~\cite{Huybrechts}, for example.) But we are very far from knowing sufficient conditions.

\section{Three theorems about compact torsion-free $\G$-manifolds} \label{sec:three-theorems}

In this final section we briefly discuss three important theorems about compact torsion-free $\G$ manifolds: an existence theorem of Joyce, the smoothness of the moduli space (also due to Joyce), and a variational characterization of compact torsion-free $\G$ manifolds due to Hitchin. Only the main ideas of the proofs are sketched. We refer the reader to the original sources for the details.

\subsection{An existence theorem for compact torsion-free $\G$ manifolds} \label{sec:existence}

In Section~\ref{sec:history} we discussed known constructions of compact irreducible torsion-free $\G$ manifolds. These constructions invoke the only analytic existence theorem that is know for torsion-free $\G$-structures, which is a result of Joyce that originally appeared in~\cite{J12} but which can also be found in~\cite[Section 11.6]{Joyce}. As mentioned in Section~\ref{sec:history}, the idea is that if one has a \emph{closed} $\G$-structure $\ph$ on $M$ whose torsion is sufficiently small, the theorem guarantees the existence of a ``nearby'' torsion-free $\G$-structure $\widetilde{\ph}$ that is in the same cohomology class as $\ph$. The statement of the theorem that we give here is a slightly modified version given in~\cite[Theorem 2.7]{JK}. 

\begin{thm}[Existence Theorem of Joyce] \label{thm:existence}
Let $\alpha$, $K_1$, $K_2$, and $K_3$ be any positive constants. Then there exist $\varepsilon \in(0,1]$ and $K_4 > 0$, such that whenever $0 < t \leq \varepsilon$, the following holds.

Let $(M, \ph)$ be a \emph{compact} $7$-manifold with $\G$-structure $\ph$ satisfying $\dd \ph = 0$. Suppose there exists a \emph{closed} $4$-form $\eta$ such that
\begin{enumerate}[(i)] \setlength\itemsep{-1mm}
\item $\| \sta_{\ph} \! \ph - \eta \|_{C^0} \leq K_1 \, t^{\alpha}$,
\item $\| \sta_{\ph} \! \ph - \eta \|_{L^2} \leq K_1 \, t^{\frac{7}{2} + \alpha}$,
\item $\| \dd (\sta_{\ph} \ph - \eta)\|_{L^{14}} \leq K_1 \, t^{-\frac{1}{2} + \alpha}$,
\item the \emph{injectivity radius} $\mathrm{inj}$ of $g_{\ph}$ satisfies $\mathrm{inj} \geq K_2 \, t$,
\item the \emph{Riemann curvature} tensor $\mathrm{Rm}$ of $g_{ph}$ satisfies $\| \mathrm{Rm} \|_{C^0} \leq K_3 \, t^{-2}$.
\end{enumerate}
Then there exists a smooth \emph{torsion-free} $\G$-structure $\widetilde{\ph}$ on $M$ such that $\| \widetilde{\ph} - \ph \|_{C^0} \leq K_4 \, t^{\alpha}$ and $[\widetilde{\ph}] = [ \ph ]$ in $H^3(M, \R)$. Here all norms are computed using the original metric $g_{\ph}$.
\end{thm}

We make some remarks about the conditions \emph{(i)--(iii)} of the theorem. Since $\ph$ is closed, it would be torsion-free if and only if $\sta_{\ph} \ph$ were also closed. The hypotheses \emph{(i)--(iii)} above say that $\sta_{\ph} \ph$ is \emph{almost closed}, in that there exists a \emph{closed} $4$-form $\eta$ that is close to $\sta_{\ph} \ph$ in various norms, namely the $C^0$, $L^2$, and (essentially) the $W^{14,1}$ norms.

The idea of the proof of Theorem~\ref{thm:existence} is as follows. Since we want $\widetilde{\ph}$ is to be in the same cohomology class as $\ph$, we must have $\widetilde{\ph} = \ph + \dd \sigma$ for some $\sigma \in \Omega^2$, and by Hodge theory we can assume that $\dd^* \sigma = 0$. Joyce shows that the torsion-free condition
\begin{equation*}
\dd \big( \sta_{\ph + \dd \sigma} (\ph + \dd \sigma) \big) = 0
\end{equation*}
can be rewritten as
\begin{equation} \label{eq:sigma}
\Delta_{\dd} \sigma = \mathcal{Q} (\sigma, \dd \sigma)
\end{equation}
where $\mathcal{Q}$ is some nonlinear expression that is at least order two in $\dd \sigma$. Joyce shows that the above equation can be solved by iteration. Explicitly, taking $\sigma_0 = 0$, then for each $k \geq 1$, Joyce solves the series of \emph{linear} equations
\begin{equation*}
\Delta_{\dd} \sigma_k = \mathcal{Q} (\sigma_{k-1}, \dd \sigma_{k-1}).
\end{equation*}
Using the \emph{a priori estimates} \emph{(i)--(iii)}, Joyce then shows that $\lim_{k \to \infty} \sigma_k$ exists as a smooth $2$-form satisfying~\eqref{eq:sigma}. This is essentially a ``fixed-point theorem'' type of argument. The complete details can be found in~\cite[Section 11.6]{Joyce}.

\subsection{The moduli space of compact torsion-free $\G$-structures} \label{sec:moduli}

Whenever one studies a certain type of structure in mathematics, it is natural to consider the ``set of all possible such structures'', modulo a reasonable notion of equivalence. Usually this ``moduli space'' of structures has its own special structure, and an understanding of the special structure on the moduli space sometimes yields information about the original object on which such structures are defined.

In our setting, consider a \emph{compact torsion-free $\G$ manifold} $(M, \ph)$. We want to consider \emph{the set of all possible torsion-free $\G$-structures} on the same underlying smooth $7$-manifold $M$, modulo a reasonable notion of equivalence. The usual notion of equivalence in differential geometry is \emph{diffeomorphism}. Indeed, if $\ph$ is a torsion-free $\G$-structure on $M$ and $F : M \to M$ is a diffeormorphism, then it is easy to see that $F^* \ph$ is also a torsion-free $\G$-structure on $M$, with metric $g_{F^* \ph} = F^* g_{\ph}$.

In fact, it is more convenient to consider only those diffeomorphisms of $M$ that are \emph{isotopic to the identity}. That is, those diffeomorphisms that are connected to the identity map on $M$ by a continuous path in the space $\mathrm{Diff}$ of diffeomorphisms of $M$. This is the \emph{connected component of the identity} in $\mathrm{Diff}$, and we denote it by $\mathrm{Diff}_0$. The reason we prefer the space $\mathrm{Diff}_0$ is because \emph{it acts trivially on cohomology}. That is, suppose $[\alpha] \in H^k (M, \R)$ and let $F \in \mathrm{Diff}_0$. Then we claim that $[F^* \alpha] = [\alpha]$. To see this, let $F_t$ be a continuous path in $\mathrm{Diff}$ with $F_0 = \mathrm{Id}_M$ and $F_1 = F$, given by the flow of the vector field $X_t$ on $M$. Since $\alpha$ is a closed form, we have
\begin{align*}
F^* \alpha - \alpha & = \int_0^1 \ddt (F_t^* \alpha) = \int_0^1 \mathcal{L}_{X_t} \alpha \\
& = \int_0^1 (\dd X_t \hk \alpha + X_t \hk \dd \alpha) = \int_0^1 \dd X_t \hk \alpha \\
& = \dd \Big( \int_0^1 X_t \hk \alpha \Big),
\end{align*}
and thus $F^* \alpha - \alpha$ is exact.

\begin{defn} \label{defn:moduli}
Let $(M, \ph_0)$ be a compact torsion-free $\G$ manifold. Let $\mathcal{T}$ be the set of all torsion-free $\G$-structures on $M$. That is,
\begin{equation*}
\mathcal{T} = \{ \ph \in \Omega^3_+ \mid \dd \ph = 0, \dd \sta_{\ph} \! \ph = 0 \}.
\end{equation*}
The \emph{moduli space} $\mathcal{M}$ of torsion-free $\G$-structures on $M$ is defined to be the quotient topological space
\begin{equation*}
\mathcal{M} = \mathcal{T} / \mathrm{Diff}_0
\end{equation*}
of $\mathcal{T}$ by the action of $\mathrm{Diff}_0$.
\end{defn}

\begin{rmk} \label{rmk:moduli}
The space $\mathcal{M}$ in Definition~\ref{defn:moduli} should probably more properly be called the \emph{Teichm\"uller space}, and then the ``moduli space'' would be the quotient $\mathcal{T} / \mathrm{Diff}$ by the full diffeomorphism group, in analogy with the usage of terminology for Riemann surfaces. However, the nomenclature we have given in Definition~\ref{defn:moduli} is standard in the field of $\G$ geometry.
\end{rmk}

The first important result that was established about the moduli space was the following theorem of Joyce, that originally appeared in~\cite{J12} but which can also be found in~\cite[Section 10.4]{Joyce}.

\begin{thm}[Moduli Space Theorem of Joyce] \label{thm:moduli}
Let $M$ be a compact $7$-manifold with torsion-free $\G$-structure $\ph_0$. The moduli space $\mathcal{M}$ of torsion-free $\G$-structures on $M$ is a \emph{smooth manifold} of dimension $b^3 = \dim H^3 (M, \R)$. In fact, the ``period map'' $\mathcal{P} : \mathcal{M} \to H^3 (M, \R)$ that takes an equivalence class $[\ph]_{\mathcal{M}}$ in the quotient space $\mathcal{M} = \mathcal{T} / \mathrm{Diff}_0$ to the deRham cohomology class $[\ph]$ is a \emph{local diffeomorphism}.
\end{thm}

The idea of the proof of Theorem~\ref{thm:moduli} is as follows. Joyce constructs a ``slice'' $\mathcal{S}_{\ph}$ for the action of $\mathrm{Diff}_0$ on $\mathcal{T}$ in a neighbourhood of any $\ph \in \mathcal{T}$. A slice $\mathcal{S}_{\ph}$ is a submanifold of $\mathcal{T}$ containing $\ph$ that is locally transverse to the orbits of $\mathrm{Diff}_0$ near $\ph$. This means that all nearby orbits of $\mathrm{Diff}_0$ each intersect $\mathcal{T}$ at only one point. Then $\mathcal{M} = \mathcal{T} / \mathrm{Diff}_0$ is locally homeomorphic in a neighbourhood of $[\ph]_{\mathcal{M}} \in \mathcal{M}$ to $\mathcal{S}_{\ph}$. Since $\ph \in \mathcal{T}$ is arbitrary, we deduce that $\mathcal{M}$ is a smooth manifold of dimension $\dim \mathcal{S}$.

In fact a slice $\mathcal{S}_{\ph}$ is given by
\begin{equation} \label{eq:slice-1}
\mathcal{S}_{\ph} = \{ \widetilde{\ph} \in \Omega^3_+ \mid \dd \widetilde{\ph} = 0, \dd \sta_{\widetilde{\ph}} \! \widetilde{\ph} = 0 , \pi_7 (\dd^* \widetilde{\ph}) = 0 \}, 
\end{equation}
where $\pi_7$ is the orthogonal projection $\pi_7 : \Omega^2 \to \Omega^2_7$ with respect to the $\G$-structure $\ph$. The way to understand where the above $\mathcal{S}_{\ph}$ comes from is to consider tangent vectors to the orbit of $\mathrm{Diff}_0$ at $\ph$. Such a tangent vector is of the form
\begin{equation*}
\left. \ddt \right|_{t=0} h_t^* \ph = \mathcal{L}_X \ph = \dd (X \hk \ph)
\end{equation*}
where $h_t$ is the flow of a smooth vector field $X$ on $M$. By the description~\eqref{eq:Omega2}, the tangent space at $\ph$ of the orbit of $\mathrm{Diff}_0$ is thus the space $\dd (\Omega^2_7)$. It thus makes sense to define
\begin{equation} \label{eq:slice-2}
\mathcal{S}_{\ph} = \{ \widetilde{\ph} \in \mathcal{T} \mid \llangle \widetilde{\ph} - \ph, \dd (X \hk \ph) \rrangle_{L^2} = 0 \, \forall X \in \Gamma(TM) \},
\end{equation}
because for $\widetilde{\ph}$ close to $\ph$, the condition of $L^2$-orthogonality to the tangent spaces of the orbit of $\mathrm{Diff}_0$ through $\ph$ would ensure local transversality. Since $\ph$ is torsion-free, we have $\dd^* \ph = 0$. Thus integration by parts shows that~\eqref{eq:slice-2} is equivalent to~\eqref{eq:slice-1}.

It still remains to explain why $\mathcal{S}_{\ph}$ is a smooth manifold of dimension $b^3$. Given $\widetilde{\ph} \in \mathcal{T}$, by Hodge theory with respect to $g_{\ph}$ we can write $\widetilde{\ph} = \ph + \xi + \dd \eta$ for some $\xi \in \mathcal{H}^3$ and some $\eta \in \dd^*(\Omega^3)$. For $\widetilde{\ph}$ sufficiently close to $\ph$ in the $C^0$ norm, Joyce shows that
\begin{equation} \label{eq:slice-3}
\widetilde{\ph} \in \mathcal{S}_{\ph} \quad \Longleftrightarrow \quad \Delta_{\dd} \eta = \sta \dd \big(  \mathcal{Q} (\xi, \dd \eta) \big)
\end{equation}
where $\mathcal{Q}$ is a nonlinear expression that is at least order two in $\dd \eta$. This is a fully nonlinear elliptic equation for $\eta$ given any $\xi \in \mathcal{H}^3$. Using the Banach Space Implicit Function Theorem, Joyce shows that the space of solutions $(\xi, \eta)$ to the right hand side of~\eqref{eq:slice-3} is a smooth manifold of dimension $b^3$. The complete details can be found in~\cite[Section 10.4]{Joyce}.

\begin{rmk} \label{rmk:moduli-structure}
A consequence of the fact from Theorem~\ref{thm:moduli} that the period map $\mathcal{P} : \mathcal{M} \to H^3(M, \R)$ is a local diffeomorphism is the following. The manifold $\mathcal{M}$ has a natural \emph{affine structure}, that is a covering by coordinate charts whose transition functions are affine maps. In Karigiannis--Leung~\cite{KL} this affine structure is exploited to study special structures on $\mathcal{M}$, including a natural \emph{Hessian metric} and a \emph{symmetric cubic form} called the ``Yukawa coupling''. This Hessian metric is obtained from the \emph{Hitchin volume functional} defined in Section~\ref{sec:variational} below.
\end{rmk}

We know \emph{very little} about the \emph{global structure} of $\mathcal{M}$. (But see the survey article by Crowley--Goette--Nordstr\"om~\cite{CGN} in the present volume for some recent progress on the (dis-)connectedness of $\mathcal{M}$.

\subsection{A variational characterization of torsion-free $\G$-structures} \label{sec:variational}

It is the case that some natural geometric structures can be given a \emph{variational interpretation}. That is, they can be characterized as critical points of a certain natural geometric functional, which means that they are solutions to the associated Euler-Lagrange equations for this functional. Some examples of such geometric structures and their associated functionals are:
\begin{itemize} \setlength\itemsep{-1mm}
\item minimal submanifolds (the volume functional),
\item harmonic maps (the energy functional),
\item Einstein metrics (the Einstein-Hilbert functional),
\item Yang-Mills connections (the Yang-Mills functional).
\end{itemize}

It was an important observation of Hitchin~\cite{Hi-7} that torsion-free $\G$-structures on compact manifolds can be given a variational interpretation. The setup is as follows. Let $M^7$ be compact, and as usual, let $\Omega^3_+$ be the set of $\G$-structures on $M$. Given $\ph \in \Omega^3_+$, we get a metric $g_{\ph}$, a Riemannian volume $\vol_{\ph}$, and a dual $4$-form $\sta_{\ph} \ph$.

\begin{defn} \label{defn:Hitchin}
The \emph{Hitchin functional} is defined to be the map $\mathcal F : \Omega^3_+ \to \R$ given by
\begin{equation} \label{eq:Hitchin}
\mathcal F (\ph) = \int_M \ph \wedge \sta_{\ph} \ph = 7 \int_M \vol_{\ph} = 7 \, \mathrm{Vol} (M, g_{\ph}),
\end{equation}
where we have used the fact that $| \ph |_{g_{\ph}}^2 = 7$ from~\eqref{eq:ph-ps-7}. Thus, up to a positive factor, $\mathcal F(\ph)$ is the total volume of $M$ with respect to the metric $g_{\ph}$.
\end{defn}

Hitchin's observation was to restrict $\mathcal{F}$ to a \emph{cohomology class} containing a closed $\G$-structure. That is, suppose $\ph_0$ is a \emph{closed} $\G$-structure on $M$, and let
\begin{equation*}
\mathcal C_{\ph} = \Omega^3_+ \cap [\ph] = \{ \widetilde{\ph} \in \Omega^3_+ \mid \dd \widetilde{\ph} = 0, [\widetilde{\ph}] = [\ph] \in H^3(M, \R) \}.
\end{equation*}
In~\cite{Hi-7}, Hitchin proved the following.

\begin{thm}[Hitchin's variational characterization] \label{thm:Hitchin}
Let $\ph$ be a closed $\G$-structure on $M$, and consider the \emph{restriction} of $\mathcal F$ to the set $\mathcal C_{\ph}$ defined above. Then $\ph$ is a \emph{critical point} of $\left. \mathcal F \right|_{C_{\ph}}$ if and only if $\ph$ is torsion-free. That is,
\begin{equation*}
\left. \ddt \right|_{t=0} \mathcal F(\ph + t \dd \eta) = 0 \quad \Longleftrightarrow \quad \dd \sta_{\ph} \! \ph = 0.
\end{equation*}
Moreover, at a critical point $\ph$, the \emph{second variation} of $\left. \mathcal F \right|_{C_{\ph}}$ is nonpositive. This means that critical points are \emph{local maxima}. 
\end{thm}

The proof of Theorem~\ref{thm:Hitchin} is quite straightforward given the following observation, which is quite useful itself in many other applications. Let $\ph(t)$ be a smooth family of $\G$-structures with $\ddt \big|_{t=0} \ph(t) = \gamma$. Then
\begin{equation} \label{eq:dTheta}
\left. \ddt \right|_{t=0} \sta_{\ph(t)} \! \ph(t) = \tfrac{4}{3} \sta \pi_1 \gamma + \sta \pi_7 \gamma - \sta \pi_{27} \gamma,
\end{equation}
where the orthogonal projections $\pi_k : \Omega^3 \to \Omega^3_k$ and the Hodge star $\sta$ are all taken with respect to $\ph(0)$. Two different proofs of~\eqref{eq:dTheta} can be found in~\cite{Hi-7} and in~\cite[Remark 3.6]{K-flows}.

The interesting observation in Theorem~\ref{thm:Hitchin} that torsion-free $\G$-structures are \emph{local maxima} of $\mathcal F$ restricted to a cohomology class motivates the idea to try to \emph{flow} to a torsion-free $\G$-structure by taking the appropriate \emph{gradient flow} of $\mathcal F$. This yields the \emph{Laplacian flow} of closed $\G$-structures. See the article by Lotay~\cite{Minischool-Lotay-flows} in the present volume for a discussion of geometric flows of $\G$-structures, including the Laplacian flow.

\end{document}